\documentclass[18pt]{article}
\usepackage{subfigure}
\usepackage{graphicx}
\usepackage{mathptmx}
\usepackage{anyfontsize}
\usepackage{t1enc}
\usepackage{amsmath, amssymb, amsbsy}
\usepackage[utf8]{inputenc} 
\usepackage[T1]{fontenc}
\usepackage{amsmath,amsfonts,amssymb}
\usepackage{subcaption}

\makeatletter
\let\@internalcite\cite
\def\cite{\def\citeauthoryear##1##2{##1, ##2}\@internalcite}
\def\shortcite{\def\citeauthoryear##1{##2}\@internalcite}
\def\@biblabel#1{\def\citeauthoryear##1##2{##1, ##2}[#1]\hfill}
\makeatother

\usepackage{amsmath}

\usepackage{amsmath}
\usepackage{amsfonts}
\usepackage{amssymb}
\usepackage[a4paper]{geometry}
\usepackage{graphicx}

\usepackage{amsmath}

\usepackage{algorithm}

\usepackage{etoolbox}
\usepackage{microtype} 

\usepackage{authblk}
\usepackage{setspace}

\usepackage{bnumexpr}

\usepackage{algorithm}

\usepackage{hyperref}

\newcommand{\R}{\mathbb{R}}

\newcommand{\quant}{\mathsf{Q}}
\newcommand{\alg}{\mathcal{A}}

\newcommand{\Xcal}{\mathcal{X}}

\allowdisplaybreaks
\newcommand{\norm}[1]{\left\lVert#1\right\rVert}

\DeclareMathOperator*{\argminB}{argmin}   

\usepackage{booktabs}
\usepackage{siunitx}

\usepackage{xspace}

\usepackage[T1]{fontenc}
\usepackage{graphicx}

{\tiny }\usepackage{amsthm}
\newtheorem{example}{Example}
\newtheorem{theorem}{Theorem}
\newtheorem{lemma}{Lemma}
\newtheorem{proposition}[theorem]{Proposition}
\newtheorem{corollary}[theorem]{Corollary}

\newtheorem{definition}{Definition}
\newtheorem{remark}{Remark}
\newtheorem{assumption}{Assumption}

\title{Uncertainty quantification in metric spaces}
\author[1,2,3]{Gábor Lugosi}
\author[4,5]{Marcos Matabuena\footnote{mmatabuena@hsph.harvard.edu}}

\affil[1]{Department of Economics and Business, Pompeu Fabra University, Barcelona, Spain}
\affil[2]{ICREA, Pg. Llu\'is Companys 23, 08010 Barcelona, Spain}
\affil[3]{Barcelona Graduate School of Economics}
\affil[4]{Department of Biostatistics, Harvard University, Boston, MA 02115, USA}
\affil[5]{Universidad de Santiago de Compostela}

\begin{document}
	\maketitle
	
	\begin{abstract}
		This paper introduces a novel uncertainty quantification framework for regression models where the response takes values in a separable metric space, and the predictors are in a Euclidean space. The proposed algorithms can efficiently handle large datasets and are agnostic to the predictive base model used. Furthermore, the algorithms possess asymptotic consistency guarantees and, in some special homoscedastic cases, we provide non-asymptotic guarantees. To illustrate the effectiveness of the proposed uncertainty quantification framework, we use a linear regression model for metric responses (known as the global Fréchet model) in various clinical applications related to precision and digital medicine. The different clinical outcomes analyzed are represented as complex statistical objects, including multivariate Euclidean data, Laplacian graphs, and probability distributions.
	\end{abstract}

	\section{Introduction}

	\subsection{Motivation and goal}
	
\noindent	The increasing use of statistical and machine learning algorithms as predictive tools is transforming  \cite{Banerji2023, HAMMOURI2023, rodriguez2022contributions}, and digital markets. Hence, it has become more critical than ever to conduct an uncertainty analysis to create trustworthy predictive models and validate their usefulness \cite{romano2019conformalized}. Data analysts tend to focus on pointwise estimation through the conditional mean between a response and a set of predictors, neglecting other crucial aspects of the conditional distribution between the involved random variables \cite{kneib2021rage}. In order to quantify the uncertainty of the point estimates, we need to estimate other characteristics of the conditional distribution beyond the conditional mean.

\noindent An innovative approach to tackle uncertainty quantification problems involves employing the conformal inference framework, first introduced by Gammerman, Vovk, and Vapnik in 1998 \cite{alex1998learning}. Building upon this foundation,  Vovk, Gammerman, and Shafer  have made significant contributions to the field \cite{vovk2005algorithmic}. Conformal inference allows one to construct prediction sets with non-asymptotic guarantees, setting it apart from other methods that only offer asymptotic guarantees. 

\noindent However, despite its attractive properties and advantages, conformal inference does have some limitations in different settings that include:

\begin{enumerate}
	\item \textbf{Computational Complexity:} The implementation of conformal inference methods can be computationally expensive, especially if data-splitting strategies to estimate the underlying regression model and prediction region (referred to as \emph{conformal-split} in the literature) are not considered \cite{vovk2018cross, solari2022multi}.
	
	\item \textbf{Conservative Intervals:} The method tends to produce conservative intervals if the regression function used is not well-calibrated with respect to the underlying model, or if the conformal inference method is mis-specified \cite{lu2023data} (e.g., using a homoscedastic model when the underlying conditional distribution function between the random response variable $Y$ and the random predictor $X$ is governed by heteroscedastic random errors).
	
	\item \textbf{Limited Applicability:} Conformal inference may have limited applicability, particularly in scenarios with multivariate responses or when the observed data is extracted from complex survey designs where the exchangeability hypothesis can be violated (see \cite{barber2022conformal}, who quantify the impact of this fact in terms of the total variation distance between distributions).
	
	\item \textbf{Asymptotic Guarantees:} In some special cases, other methods may provide stronger asymptotic guarantees compared to the algorithms derived for the conformal inference framework (see \cite{gyofi2020nearest} for a discussion).
\end{enumerate}

\noindent In the context of precision and digital medicine, clinical outcomes can take on complex statistical forms such as probability distributions or graphs \cite{rodriguez2022contributions}, and there is no general methodology available to perform uncertainty analysis in such settings. For instance, in the case of a glucose time series, a modern approach to summarizing the glucose profiles involves using a distributional representation of the time series, such as their quantile functions. As prior research has shown, these new representations \cite{ghosal2023multivariate} can capture information about glucose homeostasis metabolism that traditional diabetes biomarkers cannot measure \cite{doi:10.1177/0962280221998064,matabuena2022kernel, 10.1093/biostatistics/kxab041}.

\noindent This paper proposes a general method to quantify uncertainty in regression modeling for responses taking values in separable metric spaces \cite{AIHP_1948__10_4_215_0, petersen2019frechet, schotz2021frechet} and Euclidean predictors. The new algorithms can handle large datasets efficiently, work independently of the predictive base model fitted, and offer asymptotic consistency guarantees. We demonstrate the usefulness and applicability of the novel uncertainty quantification framework in practice through several examples in digital and precision medicine.

\noindent	Below, we introduce the notation and mathematical concepts to define the new uncertainty quantification framework.

	\subsection{Notation and problem definition}\label{sec:def}
	\noindent Let $(X,Y)\in \mathcal{X} \times \mathcal{Y}$ be a pair of random variables that play the role of the predictor and response variable in a regression model.  We assume that	$\mathcal{X}= \mathbb{R}^{p}$ and $\mathcal{Y}$ is a  separable metric space equipped with a distance $d_1$.  We assume that there exists $y\in \mathcal{Y}$ such that $\mathbb{E}(d_{1}^{2}(Y,y)) < \infty$. The regression function $m$ is defined as
	
	\begin{equation}
	m(x)= \argminB_{y\in \mathcal{Y}} \mathbb{E}(d_{1}^{2}(Y,y)|X=x),
	\end{equation}
	
	\noindent	where $x\in \mathbb{R}^{p}$.  In other words, $m$ is the conditional Fréchet mean \cite{petersen2019frechet}. For simplicity, we assume that the minimum in $(1)$ is achieved for each $x$ and moreover, the conditional Fréchet mean of $Y$ given $X=x$, is unique. We note here that one may also consider the conditional Fréchet median obtained by $\argminB_{y\in \mathcal{Y}}\mathbb{E}(d_{1}(Y,y)|X=x).$ However, this is a special case of our setup obtained by replacing the metric $d_{1}$, by $\sqrt{d_{1}}$ (which is also a metric).

	\noindent	 Suppose that $\mathcal{Y}$ is also equipped with another distance $d_{2}$. We may have $d_{2}= d_{1}$, but in some cases it is convenient to work with two distances.
	
	\noindent For $y\in \mathcal{Y}$ and $r\geq 0$, denote by $\mathcal{B}(y,r):= \{z\in \mathcal{Y};\hspace{0.1cm} d_{2}(y,z)\leq r \}$, the closed ball of center $y\in \mathcal{Y}$, and radius $r$.

	\begin{definition} \label{def:homoc} We say that the distribution of $(X,Y)$ is homoscedastic with respect to the regression function $m$ if there exists a function $\phi: [0,\infty)\to [0,1]$ such that for all $x\in \mathbb{R}^{p}$ and $r\geq 0$, 
		$$\mathbb{P}(Y\in \mathcal{B}(m(x), r)|X=x)= \phi(r).$$ 
	\end{definition}
	
	\noindent To motivate better our definition of homoscedasticity, let us consider the following examples. The first example is the most natural and commonly encountered in the statistical literature. In this case, the notion of homoscedasticity aligns with the traditional definition.
	
	\noindent However, the second example is non-trivial and serves to illustrate that even in a spaces without a vector-space structure, it is still possible to have statistical models that fall under the homoscedastic regime according to our definition.
	
	\begin{example}
		
	\noindent	
		\noindent	Suppose that $\mathcal{X}= \mathbb{R}^{p}$, $\mathcal{Y}= \mathbb{R}^{m},$ $d_{1}(\cdot,\cdot)=d_{2}(\cdot,\cdot)=\norm{\cdot-\cdot}$ (arbitrary norm in $\mathbb{R}^{m}$), and consider the regression model
		
		\begin{equation}
		Y= m(X)+\epsilon, 
		\end{equation}
		
		\noindent where, the random vector $\epsilon$, taking values in $\mathbb{R}^{m}$, is independent of $X$. 
		
		\noindent Obviously, $\mathbb{P}(d_{2}\left(Y,m\left(x\right)\right)\leq r |X=x)= \mathbb{P} \left( \norm{\epsilon}\leq r |X=x  \right)= \phi(r)$, due to the independence of $\epsilon$ and $X$.		
		
	\end{example}

\begin{example}\label{ejempldos}

	\noindent Consider $\mathcal{X} = \mathbb{R}^{p}$ and $\mathcal{Y} = \mathcal{W}_{2}(\mathbb{R})$, where $\mathcal{W}_{2}(\mathbb{R})$ denotes the 2-Wasserstein space (see \cite{MAJOR1978487, panaretos2020invitation}). More specifically,  we define $\mathcal{W}_{2}(\mathbb{R}) = \{ F \in \Pi(\mathbb{R})\}$, where $\Pi(\mathbb{R})$ is the set of distribution functions over $\mathbb{R}$ with a finite number of discontinuities.
	
\noindent	A natural metric for $\mathcal{W}_{2}(\mathbb{R})$ is $d^{2}_{\mathcal{W}_{2}}(F,G)$, is defined as 
	\begin{equation}
	d^{2}_{\mathcal{W}_{2}}(F,G) = \int_{0}^{1} (Q_{F}(t) - Q_{G}(t))^{2} dt.
	\end{equation}

\noindent	Here, $Q_{F}$ and $Q_{G}$ denote the quantile functions of the distribution functions $F$ and $G$, respectively.
	
\noindent Define the discrete-time stochastic process $Z(\cdot)$ as
\begin{equation*}
Z(j) = g(X) + \epsilon(j), \quad j = 1, \dots, n,
\end{equation*}
\noindent where $X$ is a random vector taking values in $\mathbb{R}^{p}$, $g: \mathbb{R}^{p} \to \mathbb{R}$ is a continuous function, and for each $t = 1, \dots, n$, $\epsilon(t) \sim \mathcal{N}(0, \sigma^{2}_{\epsilon})$, with $\epsilon(j) \perp \epsilon(j')$ for $j \neq j'$, where $\perp$ denotes the statistical independence between two random variables, and $\epsilon(1),\dots, \epsilon(n)$ are independent of $X$.
	
\noindent	Then, we define the element $F_{n} \in \mathcal{W}_{2}(\mathbb{R})$ as
	\begin{equation}
	F_{n}(t) = \frac{1}{n} \sum_{j=1}^{n} \mathbb{I}\{Z(j) \leq t\}.
	\end{equation}

	
	
	

\noindent	Similarly, we can define the $\rho$-th quantile related to $F_{n}$ as $Q_{n}(\rho) = \inf \{t: F_{n}(t) \geq \rho\}$, which has an explicit expression as $Q_{n}(\rho) = g(X) + \sum_{j: [j/n]\leq \rho} \epsilon_{(j)}$, where $\epsilon_{(j)}$ is the $j$-th ordered random error value from $\{\epsilon(j)\}_{j=1}^{n}$. The expectation of the quantile function can be easily derived as:
\begin{equation}
\mathbb{E}(Q_{n}(\rho)\mid X) = g(X) + \sigma_{\epsilon} h_{n}(\rho),
\end{equation}

\noindent where $h_{n}(\cdot)$ is a function that depends on the sample size $n$. To simplify the notation, we denote $Q_{n}= \left[F_{n}\right]^{-1}$. 

\noindent By the definition of the metric $d^{2}_{\mathcal{W}_{2}}(\cdot,\cdot)$, $\left[m(X)\right]^{-1}(\rho)= \mathbb{E}\left[Q_{n}(\rho)|X\right]$. To see this, we note that

\begin{equation}
m(x)= \argminB_{y\in \mathcal{W}_{2}(\mathbb{R})} \mathbb{E}(d^{2}_{\mathcal{W}_{2}}(F_{n},y)\mid X=x)=  \argminB_{y\in \mathcal{W}_{2}(\mathbb{R})} \mathbb{E}\left( \int_{0}^{1}\left(Q_{n}(t)-y^{-1}(t))^{2}dt\right)\mid X=x\right) = \mathbb{E}\left[Q_{n}\mid X=x\right]^{-1},
\end{equation}

\noindent where, in the final step, we leverage the property that the mean operator minimizes the aforementioned optimization problem in any Euclidean space.

\noindent Then,	the squared Wasserstein distance between the distribution function $F_{n}$ with respect to its conditional Fréchet mean $m(X)$ is calculated as:
	
	\begin{equation*}
	d^{2}_{\mathcal{W}_{2}}(F_{n},m(X)) = \int_{0}^{1}\left( g(X)+ \sum_{j: [j/n] \leq \rho} \epsilon_{(j)}-g(X)-\sigma_{\epsilon}h_{n}(\rho)\right)^{2}d\rho,
	\end{equation*}
	
\noindent	which does not depend on the value of $X$.

\end{example}

	\noindent We say that the distribution of $(X,Y)$ is heterocedastic with respect to the regression function $m$  when
	$$\mathbb{P}(Y\in  \mathcal{B}(m(x),r)|X=x)$$ 
	\noindent	 may depend on  $x$.
		
	\noindent	Given a random observation $X$, and confidence level $\alpha\in [0,1]$, our  goal is  to estimate a prediction region $C^{\alpha}(X)\subset \mathcal{Y}$ of the response variable $Y$  that  contains the response variable with probability at least $1-\alpha$. We  introduce the population version of the  prediction region.

	\begin{definition}
	\noindent		We define the oracle-prediction region as   
		\begin{equation} 
		C^{\alpha}(x):= \mathcal{B}(m(x),r(x)),
		\end{equation}
		\noindent where $r(x)$ is the smallest number $r$ such that $\mathbb{P}(Y\in \mathcal{B}(m(x),r)|X=x)\geq 1-\alpha$. 
	\end{definition}

	\noindent In order to estimate	$ C^{\alpha}(x)$,	suppose that we observe a random sample $\mathcal{D}_{n}=\{ (X_i,Y_i)\in \mathcal{X}\times \mathcal{Y}: i\in [n]:= \{1,2,\dots,n\}\}$ containing independent, identically distributed (i.i.d.) pairs drawn from the same distribution as $(X,Y)$.  In the rest of this paper, we  split  $\mathcal{D}_{n}$ in two disjoint subsets,  $\mathcal{D}_{n}=\mathcal{D}_{train} \cup \mathcal{D}_{test}$, in order to estimate the center and radius of the ball with two independent random samples.   We denote the set of indexes $[S_1]:= \{ i\in [n]:  \left(X_i,Y_i\right)\in   \mathcal{D}_{train} \}$, $[S_2]:= \{ i\in [n]:  (X_i,Y_i)\in   \mathcal{D}_{test} \}$, and we denote  $|\mathcal{D}_{train}|= n_1$ and $|\mathcal{D}_{test}|= n_2= n-n_1$.	
	
	\noindent	We propose estimators of the oracle-prediction region $C^{\alpha}(x)$ of the form $\widetilde{C}^{\alpha}(x)=\mathcal{B}(\widetilde{m}(x),\widetilde{r}_{\alpha}(x))$, where $\widetilde{m}(x)$ and $\widetilde{r}_{\alpha}(x)$ are estimators of $m(x)$ and $r_{\alpha}(x)$. For simplicity, we assume that the outcomes of the estimator $\widetilde{m}(\cdot;\mathcal{D}_{\text{train}})$ remain unaffected by the order in which the random elements of $\mathcal{D}_{\text{train}}$ are considered during the estimation process (i.e., the estimator is permutation-invariant). This technical requirement is commonly introduced in the context of conformal inference algorithms.

	\begin{definition}
	Consider a measurable mapping $C: \mathcal{X} \to 2^{\mathcal{Y}}$. We define the error of $C$ with respect to the  oracle-prediction region as
		
		\begin{equation*}
		\varepsilon(C,x)= \mathbb{P}(Y\in C(X)\triangle C^{\alpha}(X)\mid X=x), 
		\end{equation*}

		\noindent where $\triangle$ denotes the symmetric diference.

		\noindent	If $C= \widetilde{C}$, is constructed from $\mathcal{D}_{n}$, then  
		
		\begin{equation*}
		\varepsilon(\widetilde{C},x)= \mathbb{P}(Y\in \widetilde{C}(X)\triangle C^{\alpha}(X)\mid X=x,\mathcal{D}_{n}).	
		\end{equation*} 
		
		\noindent We are interested in the integrated error

		\begin{equation*}
		\mathbb{E}(\varepsilon(\widetilde{C},X)|\mathcal{D}_{n}).
		\end{equation*}

		\noindent We say that $\widetilde{C}$ is consistent if

		\begin{equation}
		\lim_{n\to \infty} \mathbb{E}(\varepsilon(\widetilde{C},X)|\mathcal{D}_{n})\to 0  \text{ in probability. }
		\end{equation}
	\end{definition}

	\subsection{Contributions}
	
	Now, we summarize the main contributions of this paper: 
	
	\begin{enumerate}
		\item We propose two uncertainty quantification algorithms, one for the homoscedastic case and another one for the general case.
		\begin{enumerate}
			\item  Homoscedastic case: We  propose a natural generalization of split conformal inference techniques \cite{vovk2018cross, solari2022multi} in the context of metric space responses. We obtain non-asymptotic guarantees of marginal coverage of the type $\mathbb{P}(Y\in\widetilde{C}^{\alpha}(X))\geq 1-\alpha$. In this setting, we also show that $\widetilde{C}$ is consistent,	under mild assumptions, which only require the consistency of the conditional Fréchet mean estimator, i.e., $\mathbb{E}(d_{2}(\widetilde{m}(X),m(X))| \mathcal{D}_{train})\to 0$ in probability as $|\mathcal{D}_{train}|\to \infty$.
			
			\item  Heteroscedastic case: We present a novel local algorithm based on the concept of $k$-nearest neighbors, initially proposed by Fix and Hodges \cite{fix1951discriminatory} and later extended by Cover \cite{cover1968estimation} (for a contemporary reference on the topic, see \cite{biau2015lectures}). While we do not provide non-asymptotic guarantees regarding the marginal coverage of the prediction region, we prove its consistency.

			\item The predictions regions are the generalization of the concept of conditioned quantiles for metric space responses. The prediction regions are defined by a center function $m$ and a specific radius $r$ corresponding to a chosen probability level, $\alpha\in [0,1]$. This extension builds upon recent advances in the theory of unconditional quantiles for metric spaces \cite{liu2022quantiles} and provides a natural generalization of the existing concepts.
			
		\end{enumerate}
  
Importantly, from a computational perspective, after estimating the regression function $m$,
 the uncertainty quantification algorithms are capable of handling problems involving millions of data points in just a few seconds.

	\item 	We introduce an  approach to discern between homoscedastic and heteroscedastic uncertainty quantification models using a global hypothesis testing criterion. Unlike classical approaches \cite{goldfeld1965some}, our methods take a general and non-parametric perspective, allowing for wider applicability. We draw upon the concept distance correlation, based on the energy distance initially introduced \cite{10.1214/009053607000000505, szekely2017energy}. Our criterion provides a powerful tool to assess the homoscedasticity or heteroscedasticity of the underlying distribution.

 \item The proposed uncertainty quantification framework has the potential to develop new models for complex statistical objects for different modeling tasks. In this paper, we focus on a \textit{local} \cite{https://doi.org/10.1111/biom.12254} and \textit{general} variable selection method designed specifically for responses in an arbitrary separable metric space. Other existing variable selection methods are limited to linear models and adopt a global perspective \cite{tucker2021variable}.
	
	\item We illustrate the potential of the proposed framework in real-world scenarios and demonstrate its performance in applications in the domain of digital and precision medicine \cite{tu2020era}. The different  applications involve probability distributions with the $2$-Wasserstein metric, fMRI data with Laplacian graphs, and multivariate Euclidean data. As application of our framework, we discuss an application that provides normative reference values \cite{wang2018hand} for glucose time series in healthy individuals using the recently proposed distributional glucose time series representations \cite{doi:10.1177/0962280221998064}. This new analysis has the potential to establish clearer characterizations as the glucose values evolve in the healthy population conditioned on patient characteristics, providing a formal basis for establishing personalized definition of diseases. For example, the derivation of reference values across healthy populations has gained increasing interest among  digital medical researchers, as demonstrated by the study \cite{Shapiro2023}. In their research, they defined normative values for a digital biomarker, namely, the "Pulse Oximetry value."

	\end{enumerate}

\subsection{Literature overview}
	
\subsubsection{Uncertainity quantification}\label{sec:summary}

\noindent In recent years, uncertainty quantification became an  active research area  \cite{geisser2017predictive, politis2015model}. The impact of uncertainty quantification on data-driven systems has led to a remarkable surge of interest in both applied and theoretical domains. These works delve into the profound implications of uncertainty quantification in statistical science and beyond such as in the biomedical field \cite{HAMMOURI2023, Banerji2023}.

\noindent Geisser's pioneering book develops a mathematical theory of prediction inference \cite{geisser2017predictive}. Building upon Geisser's foundations, Politis presented a comprehensive methodology that effectively harnesses resampling techniques \cite{politis2015model}. Additionally, the  book of Vovk, Gammmerman,  and Shafer \cite{vovk2005algorithmic} (see  recent reviews \cite{angelopoulos2021gentle,fontana2023conformal} here) 
has been influential.

\noindent One of the most widely used and robust frameworks for quantifying uncertainty in statistical and machine learning models is conformal inference \cite{shafer2008tutorial}. The central idea of conformal inference is rooted in the concept of exchangeability \cite{kuchibhotla2020exchangeability}. For simplicity, we assume that the random elements observed, $\mathcal{D}_{n}$,  are independent and identically distributed (i.i.d.).
\noindent Now, we present a general overview of conformal inference methods for regresion models with scalar responses. Consider the sequence $\mathcal{D}_{n} = \{(X_i, Y_{i})\}^{n}_{i=1}$ of i.i.d. random variables. Given a new i.i.d. random pair $(X, Y)$ with respect to $\mathcal{D}_{n}$, conformal prediction, as introduced by \cite{vovk2005algorithmic}, provides a family of algorithms for constructing prediction intervals independently of the regression algorithm used.

\noindent Fix any  regression algorithm 
\[\alg: \ \cup_{n\geq 0} \left(\Xcal\times \R\right)^n \ \rightarrow \
  \left\{\textnormal{measurable functions $\widetilde{m}: \Xcal\rightarrow\R$}\right\}, 
\] 
which maps a data set containing any number of pairs $(X_i,Y_i)$, to a fitted
regression function $\widetilde{m}$. The algorithm $\alg$ is required to treat
data points symmetrically, i.e.,

\begin{equation}\label{eqn:alg_symmetric}
\alg\big((x_{\pi(1)},y_{\pi(1)}),\dots,(x_{\pi(n)},y_{\pi(n)})\big) =
\alg\big((x_1,y_1),\dots,(x_n,y_n)\big)
\end{equation}
for all $n \geq 1$, all permutations $\pi$ on $[n]=\{1,\dots,n\}$, and all
$\{(x_i,y_i)\}_{i=1}^{n}$. Next, for each $y\in\R$, let \[
\widetilde{m}^{y} = \alg\big((X_1,Y_1),\dots,(X_n,Y_n),(X,y)\big)
\]
denote the trained model, fitted to the training data together with a
hypothesized test point $(X,y)$, and let  
\begin{equation}\label{eqn:R_y_i}
R^y_i = 
\begin{cases}|Y_i - \widetilde{m}^{y}(X_i)|, & i=1,\dots,n,\\ 
|y-\widetilde{m}^{y}(X)|, & i=n+1.
\end{cases}
\end{equation}
The prediction interval for $X$ is then defined as 
\begin{equation}\label{eqn:def_fullCP}
\widetilde{C}^{\alpha}(X;\mathcal{D}_{n}) =\left\{y \in\R \ : \ R^y_{n+1}\leq
  \quant_{1-\alpha}\left(\sum_{i=1}^{n+1} \tfrac{1}{n+1} \cdot \delta_{R^y_i}
  \right)\right\}, 
\end{equation}

\noindent  where $\quant_{1-\alpha}\left(\sum_{i=1}^{n+1} \tfrac{1}{n+1} \cdot \delta_{R^y_i}
  \right)$ denotes the quantile  of order $1-\alpha$ of the empirical distribution $\sum_{i=1}^{n+1} \tfrac{1}{n+1} \cdot \delta_{R^y_i}$.

\noindent The full conformal method is known to guarantee distribution-free
predictive coverage at the target level $1-\alpha$: 

\begin{theorem}[Full conformal prediction
  \cite{vovk2005algorithmic}]\label{thm:background_fullCP} 
If the data points $(X_1,Y_1),\dots,(X_n,Y_n),(X,Y)$ are i.i.d.\ (or
more generally, exchangeable), and the algorithm $\alg$ treats the input data
points symmetrically as in~\eqref{eqn:alg_symmetric}, then the full conformal
prediction set defined in~\eqref{eqn:def_fullCP} satisfies 
\[
\mathbb{P}(Y\in \widetilde{C}^{\alpha}(X; \mathcal{D}_{n})) \geq 1-\alpha.
\]
\end{theorem}
\noindent The same result holds  for split conformal methods, which involve estimating the regression function on $\mathcal{D}_{\text{train}}$ and the quantile on $\mathcal{D}_{\text{test}}$.

\noindent Conformal inference (both split and full) was
initially proposed in terms of ``nonconformity scores''
\smash{$\widehat{S}(X_i,Y_i)$}, where \smash{$\widehat{S}$} is a fitted function
that measure the extent to which a data point $(X_i,Y_i)$ is unusual relative to
a training data set $\mathcal{D}_{train}$. For simplicity, so far we have only presented the
most commonly used nonconformity score, which is the residual from the fitted
model
\begin{equation}\label{eqn:standard_nonconformity_score}
\widehat{S}(X_i,Y_i) := |Y_i - \widetilde{m}(X_i)|,
\end{equation}
\noindent \noindent \noindent where $\widetilde{m}$ is  estimator of the regression function, estimated using random elements from the first split, $\mathcal{D}_{train}$.

\noindent The non-asymptotic guarantees provided by Theorem~\ref{thm:background_fullCP} can be achieved under  more general  conditions. For instance, Vovk \cite{vovk2012conditional} established non-asymptotic guarantees with respect to the conditional coverage probability $\mathcal{D}_{n}$, employing the same hypotheses used in Theorem~\ref{thm:background_fullCP}. He proves inequalities of the form

\begin{equation}
\mathbb{P}(Y \in \widetilde{C}^{\alpha}(X; \mathcal{D}_{n}) \mid \mathcal{D}_{n})\geq 1-\alpha,
\end{equation}

\noindent while Barber et al. \cite{foygel2021limits} investigated the case of conditional coverage probability with respect to a specific point $x \in \mathcal{X}$, meaning

\begin{equation}\label{eqn:cond}
 \mathbb{P}(Y \in \widetilde{C}^{\alpha}(X; \mathcal{D}_{n}) \mid X=x)\geq 1-\alpha.
\end{equation}

\noindent However, only in some special cases, such as those involving a finite predictor space, is it possible to generalize the results of Theorem \ref{thm:background_fullCP}.

\noindent In the general i.i.d. context, and beyond the conformal inference framework, within the asymptotic regime, Györfi and Walk \cite{gyofi2020nearest} presented a kNN-based uncertainty quantification algorithm conditioned on both: i) a random sample $\mathcal{D}_{n}$, and ii) covariate characteristics. This algorithm possesses the property that 

\begin{equation*}
\lim_{n\to \infty} \mathbb{P}(Y \in \widetilde{C}^{\alpha}(X; \mathcal{D}_{n})| \mathcal{D}_{n}, X=x)= 1-\alpha \text{ in probability}.
\end{equation*}

\noindent The work of \cite{gyofi2020nearest} is particularly relevant as it provides optimal non-parametric rates for kNN-based uncertainty quantification.

\noindent Several authors  established non-asymptotic guarantees for the coverage outlined in equation \eqref{eqn:cond}. However, in general, this is impossible without incorporating strong assumptions about the joint distribution of $(X, Y)$ or possessing exact knowledge of the distribution \cite{foygel2021limits, gibbs2023conformal}.

\noindent \cite{lei2014distribution} introduced  conditions under which the  condition in ~\eqref{eqn:cond} can lead to the Lebesgue measure of the estimated prediction region $\widetilde{C}^{\alpha}(X; \mathcal{D}_{n})$ becoming unbounded under general hypotheses. More specifically, they prove that,

\begin{equation*}
\lim_{n \to \infty} \mathbb{E}[\text{Leb}(\widetilde{C}^{\alpha}(X; \mathcal{D}_{n}))] = \infty,
\end{equation*}

\noindent where $\text{Leb}$ denotes the Lebesgue measure.

	\noindent Conformal inference techniques have been applied to various regression settings, including the estimation of the conditional mean regression function \cite{lei2018distribution}, conditional quantiles \cite{https://doi.org/10.1002/sta4.261}, and different quantities that arise from conditional distribution functions \cite{singh2023distribution, chernozhukov2021distributional}. In recent years, multiple extensions of conformal techniques have emerged to handle counterfactual inference problems \cite{chernozhukov2021exact, https://doi.org/10.1111/rssb.12445, yin2022conformal, jin2023sensitivity}, heterogeneous policy effect \cite{cheng2022conformal}, reinforcement learning \cite{dietterichreinforcement}, federated learning \cite{lu2023federated}, outlier detection \cite{bates2023testing}, hyphotesis testing \cite{hu2023two}, robust optimization \cite{johnstone2021conformal},  multilevel structures \cite{fong2021conformal, dunn2022distribution}, missing data \cite{matabuena2022kernel, zaffran2023conformal}, and survival analysis problems \cite{https://doi.org/10.1111/rssb.12445, teng2021t}, as well as problems involving dependent data such as time series and spatial data \cite{chernozhukov2021exact,  xu2021conformal, sun2022conformal, xu2023conformal}.

	\noindent	Despite the significant progress in conformal inference methods, there has been relatively little work on handling multivariate \cite{johnstone2022exact, messoudi2022ellipsoidal} and functional data \cite{doi:10.1080/01621459.2012.751873,diquigiovanni2022conformal,matabuena2021covid, dietterich2022conformal,ghosal2023multivariate}. In the case of classification problems, the first contributions to handle multivariate responses have also been recently proposed (see, for example, \cite{cauchois2021knowing}).

\noindent  \cite{10.1093/imaiai/iaac017,wu2023bootstrap, das2022model}, provide asymptotic marginal guarantees and, in certain cases, asymptotic results for the conditional coverage. The key idea behind these approaches is the application of resampling techniques, originally proposed by Politis et al. \cite{politis1994large}, to residuals or other score measures. Notably, these methodologies are applicable regardless of the predictive algorithm being used, as emphasized by \cite{politis2015model}.
	
	\noindent	Bayesian methods are also an important framework to quantify uncertainty (see \cite{chhikara1982prediction}), which can also be integrated with conformal inference methods \cite{angelopoulos2021gentle, fong2021conformal, patel2023variational}. Gaussian response theory, including linear Gaussian and multivariate response regression models \cite{thombs1990bootstrap, doi:10.1080/00031305.2022.2087735}, as a particular case, is a classical and popular approach. However, the latter techniques generally introduce stronger parametric assumptions in statistical modeling or include the limitation or difficulty in selecting the appropriate prior distribution in Bayesian modeling \cite{gawlikowski2023survey}. The theory of tolerance regions gives another connection with the problem studied here \cite{fraser1956tolerance, hamada2004bayesian}, which was generalized for the multivariate case with the notion of depth bands (see  \cite{li2008multivariate}). However, a few conditional depth measures are available in the literature \cite{garcia2023causal}. Depth band measures for statistical objects that take values in metric spaces have recently been proposed \cite{geenens2021statistical, dubey2022depth, liu2022quantiles, virta2023spatial} but only in the unconditional case.

	\subsubsection{Statistical modeling in metric spaces}
	
\noindent	One of the most prominent applications of statistical modeling in metric spaces is in biomedical problems \cite{rodriguez2022contributions}. In personalized and digital medicine applications, it is increasingly common to measure patients' health conditions with complex statistical objects, such as curves and graphs, which allow recording patients' physiological functions and measuring the topological connectivity patterns of the brain at a high resolution level. For example, in recent work, the concept of "glucodensity"  \cite{doi:10.1177/0962280221998064} has been coined, which is a distributional representation of a patient's glucose profile that improves existing methodology in diabetes research \cite{matabuena2022kernel}. This representation is also helpful in obtaining better results with accelerometer data \cite{matabuena2021distributional,ghosal2021scalar,matabuena2022physical, ghosal2023multivariate}.
	
	\noindent From a methodological point of view, statistical regression analysis of response data in metric spaces is a novel research direction \cite{fan2021conditional,chen2021wasserstein,petersen2021wasserstein, zhou2021dynamic, dubey2022modeling, 10.3150/21-BEJ1410, 10.3150/21-BEJ1410, kurisumodel,chen2023sliced}. The first papers on hypothesis testing \cite{10.1214/20-AOP1504,dubey2019frechet,petersen2021wasserstein,fout2023fr}, variable selection \cite{tucker2021variable}, missing data \cite{matabuena2024personalized}, 
multilevel models \cite{matabuena2024multilevel,
bhattacharjee2023geodesic}, dimension-reduction \cite{zhang2022nonlinear},  semi-parametric regression models \cite{bhattacharjee2021single, ghosal2023predicting}, semi-supervised algorithms \cite{	qiu2024semisupervised}
 and non-parametric regression models \cite{schotz2021frechet,hanneke2022universally, bulte2023medoid, bhattacharjee2023nonlinear} have recently appeared.

	\section{Mathematical models}

\noindent In this section, we present a framework for uncertainty quantification, building upon the previous definition of the oracle prediction region (see Section \ref{sec:def}):

	\begin{equation}
	C^{\alpha}(x):= \mathcal{B}(m(x),r(x)).
	\end{equation}
	
	\noindent Recall that $r(x)$ is the smallest number $r$ such that $\mathbb{P}(Y\in \mathcal{B}(m(x),r)\mid X=x)\geq 1-\alpha$.
	
	\noindent The initial phase of our  uncertainty quantification methodology entails the estimation of the Fréchet regression function $m:\mathcal{X}\to \mathcal{Y}$.  This estimation process leverages the  geometry introduced by the distance function $d_{1}$.

\noindent In the second step, we determine the radius that governs the number of points encompassed within the prediction regions. This selection is made to ensure that the set includes at least $(1-\alpha)$ proportion of points from the dataset.

\noindent We emphasize that our method establishes for different confidence levels $\alpha \in [0,1]$, different level sets for the prediction regions problem. This approach allows us to associate the prediction regions with the notion of conditional quantiles for metric space responses, generalizing the existing marginal approaches \cite{liu2022quantiles}.

	\subsection{Homoscedastic case}\label{sec:homo}

\noindent Algorithm \ref{alg:metd1} outlines the core steps of our uncertainty quantification method for the homoscedastic case. The key idea behind our approach is to estimate the function \(m\) and the radius \(r\) by dividing the data into two independent datasets.

\noindent In the homoscedastic case, the distribution of $ d_{2}(Y, m(x))$ remains invariant regardless of the point \(x \in \mathcal{X}\). As a result, it is natural to which estimate the radius using the empirical distribution from the random sample  \(\mathcal{D}_{test}= \{d_{2}(Y_i, \widetilde{m}(X_i))\}_{i \in [S_2]}\).  Then, we can formulate a conformalized version of this algorithm to attain non-asymptotic guarantees of the form $\mathbb{P}(Y\in \widetilde{C}^{\alpha}(X;\mathcal{D}_{n}))\geq 1-\alpha$.

\noindent The core idea of estimating different quantities of the model in different splits is not new in the literature of conformal inference. This approach enhances the computational feasibility of the methods, providing non-asymptotic guarantees and facilitating theoretical asymptotic analysis. However, there may be some loss in statistical efficiency due to estimating model quantities in subsamples. These methods, known as split conformal methods, have been extensively studied (see, for example, \cite{vovk2018cross, solari2022multi}).

		\begin{algorithm}[ht!]
		\caption{Uncertainty quantification algorithm homoscedastic set-up}
		\begin{enumerate}
			\item Estimate the function $m(\cdot)$, by  $\widetilde{m}(\cdot)$ using the random sample $\mathcal{D}_{train}$.
			\item For all $i\in [S_2]$, evaluate $\widetilde{m}(X_i)$ and define $\widetilde{r}_i= d_{2}(Y_i, \widetilde{m}(X_i))$.
			\item Estimate the empirical distribution $\widetilde{G}^{*}(t)
			=   \frac{1}{n_2}  \sum_{i\in [S_2]} \mathbb{I}\{\widetilde{r}_i \leq t\}$ and denote by $\widetilde{q}_{1-\alpha}$ the empirical quantile of level $1-\alpha$ of the empirical distribution function  $\widetilde{G}^{*}(t)$.
			\item Return  $\widetilde{C}^{\alpha}(x;\mathcal{D}_{n})= \mathcal{B} (\widetilde{m}(x),\widetilde{q}_{1-\alpha})$. 
		\end{enumerate}\label{alg:metd1}
	\end{algorithm}

\noindent Observe that we have the marginal finite sample guarantee. 
	
	\begin{proposition}	\label{th:finite}	    
For any regression function \(m\) and estimator \(\widetilde{m}\), assuming the random elements $\{(X_i,Y_i) \}_{i=1}^{n}$ are i.i.d. with respect to a pair \((X,Y)\), Algorithm \ref{alg:metd1} satisfies:
		\begin{equation*}
	 \mathbb{P}(Y \in \widetilde{C}^{\alpha}(X; \mathcal{D}_{n}))\geq 1-\alpha.
		\end{equation*}
	\end{proposition}

\begin{proof}    
\noindent This follows by observing that

\begin{equation*}
\begin{aligned}
    \mathbb{P}(Y \in \widetilde{C}^{\alpha}(X; \mathcal{D}_{n})) &= \mathbb{P}(Y \in \mathcal{B}(\widetilde{m}(X), \widetilde{q}_{1-\alpha})) \\
    &= \mathbb{E}\left[\mathbb{I}\{ d_{2}(Y,\widetilde{m}(X)) \leq \widetilde{q}_{1-\alpha} \}\right] \geq 1-\alpha.
\end{aligned}
\end{equation*}
\end{proof}
	
\begin{remark}
\noindent Proposition \ref{th:finite} holds in both the homoscedastic case  and the heteroscedastic case. However, when we apply Algorithm \ref{alg:metd1} to a heteroscedastic model, the resulting conditional coverage can substantially deviate from the true value. Under such circumstances, the algorithm may not be consistency.

\end{remark}

	\begin{remark}
		
\noindent	In Algorithm \ref{alg:metd1}, when handling discrete spaces, it becomes crucial to incorporate randomization strategies in the radius estimation. This ensures that the random variable $d_{2}(Y,m(X))$ remains continuous, and guarantees that for any two distinct $\widetilde{r}_{j}$ and $\widetilde{r}_{j'}$, the probability of them being equal is precisely zero (with probability one).

\noindent	This general strategy has been extensively explored, including \cite{kuchibhotla2020exchangeability} (see Definition $1$), and \cite{cauchois2021knowing}, which specifically focuses on discrete structures. The integration of randomization techniques in such scenarios ensures the continuity of the random variable and leads to:

	\begin{equation*}
1-\alpha + \frac{1}{n_{2}+1} \geq \mathbb{P}(Y \in \widetilde{C}^{\alpha}(X; \mathcal{D}_{n})) \geq 1-\alpha.
\end{equation*}

	\end{remark}

	\subsubsection{Statistical consistency theory}

	Next we introduce the technical assumptions that ensure the asymptotic consistency of the Algorithm \ref{alg:metd1}.
	
	\begin{assumption}
		Suppose that the following hold:
		\begin{enumerate}\label{1Smet}
			\item $n_2\to \infty.$
			\item $\widetilde{m}$ is a consistent estimator in the sense that $\lim_{n_{1}\to \infty}\mathbb{E}(d_{2}(\widetilde{m}(X),m(X))| \mathcal{D}_{train})\to 0$, in probability.
 \item  The population quantile \(q_{1-\alpha} = \inf\{t \in \mathbb{R} : G(t) = \mathbb{P}(d_{2}(Y, m(X)) \leq t) = 1-\alpha\}\) exists uniquely and is a continuity point of the function \(G(\cdot)\).
 \end{enumerate}
	\end{assumption}


	\begin{theorem}\label{th:conshom}
		Under Assumption \ref{1Smet}, the estimated prediction region $\widetilde{C}^{\alpha}(x;\mathcal{D}_{n})$ obtained using Algorithm \ref{alg:metd1} satisfies

		\begin{equation}
		\lim_{n\to \infty} \mathbb{E}(\varepsilon(\widetilde{C}^{\alpha}(X;\mathcal{D}_{n}))\to 0  \text{ in probability.}  
		\end{equation}
		
	\end{theorem}
	
	\begin{proof}
		See Appendix. 
	\end{proof}

\begin{assumption}\label{1S2met}
Suppose that for all \(t \in \mathbb{R}\), the distribution function of distances \(G(t, x) = \mathbb{P}(d_2(Y, m(x)) \leq t \mid X = x)\) is uniformly Lipschitz  with constant \(L\) in the sense that for all \(t, t' \in \mathbb{R}\) and for all \(x \in \mathcal{X}\), we have \(|G(t, x) - G(t', x)| \leq L |t - t'|\).
\end{assumption}

\begin{example}
\noindent	
Suppose that $\mathcal{X}= \mathbb{R}^{p}$, $\mathcal{Y}= \mathbb{R}^{m},$ $d_{1}(\cdot, \cdot)=d_{2}(\cdot,\cdot)=\norm{\cdot-\cdot}$ ($L^{2}$ norm), and consider the regression model
		
		\begin{equation}\label{eqn:gen}
		Y= m(X)+\epsilon, 
		\end{equation}
	\noindent	  where $\epsilon$ is a random vector taking values in $\mathbb{R}^{m}$, and independent from $X$. Let $f$ be the density function of the random variable $\epsilon$. If  the density function $f$ is bounded,  then, $G$ is uniformly Lipschitz.
  \end{example}
    
 \begin{remark}
Suppose that the assumptions of Example $3$ are satisfied. Consider any estimator \(\widetilde{m}\) of \(m\), and define the distribution of the distances:
\[
G^{*}(t, x) = \mathbb{P}(d_{2}(Y, \widetilde{m}(X)) \leq t \mid X = x).
\] 
We can express this probability in terms of the perturbation introduced by the estimator:
\[
G^{*}(t, x) = \mathbb{P}(\| \epsilon + (m(X) - \widetilde{m}(X)) \| \leq t \mid X = x).
\] 
Then for all $x$, the function \(G^{*}(t, x)\) is Lipschitz continuous with a constant \(L\).
\end{remark}

	\begin{proposition}\label{the:minimax1}
		Under Assumptions \ref{1Smet} and \ref{1S2met},  and assuming that $G^{*}(t,x)=\mathbb{P}(d_{2}(Y,\widetilde{m}(X))\leq t\mid X=x),$  is uniformly Lipschitz,  we can bound the conditional expected error of $\widetilde{C}$ given $\mathcal{D}_{n}$ as the sum of two terms: a function of the expected distance between $\widetilde{m}(X)$ and $m(X)$, conditional on the training set $\mathcal{D}_{train}$, and a function of the difference between the estimated and true $\alpha$-quantiles of the distribution of radius $r$, denoted by $\widetilde{q}_{1-\alpha}$ and ${q}_{1-\alpha}$, respectively. Specifically, we can write:
		
		\begin{equation*}
		\mathbb{E}(\varepsilon(\widetilde{C},X)|\mathcal{D}_{n})\leq C(\mathbb{E}(d_{2}(\widetilde{m}(X),m(X))|\mathcal{D}_{train}) + 4(|\widetilde{q}_{1-\alpha}-{q}_{1-\alpha}|)),
		\end{equation*}
		
		\noindent where $C$ is a positive constant depending  on the Lipschitz constants of the functions $G$ and $G^{*}$.
  \end{proposition}
		
	\begin{proof}
		See Appendix. 
	\end{proof}

\begin{example}
\noindent
Suppose that $\mathcal{X} = \mathbb{R}^p$, $\mathcal{Y} = \mathbb{R}^m$, and $d_1(\cdot, \cdot) = d_2(\cdot, \cdot) = \|\cdot - \cdot\|$ (the $L^{2}$ norm). Consider the regression model
\begin{equation}\label{reg:def}
Y = m(X) + \epsilon,
\end{equation}
\noindent where $\epsilon$ is a random vector taking values in $\mathbb{R}^m$ and is independent of $X$. Let $f$ denote the density function of $\epsilon$. In 
the multivariate scale-localization regression models defined in  \eqref{reg:def}, we can obtain the rates of Proposition \ref{the:minimax1}

	\begin{equation*}
		\mathbb{E}(\varepsilon(\widetilde{C},X)|\mathcal{D}_{n})\leq C(\mathbb{E}(d_{2}(\widetilde{m}(X),m(X))|\mathcal{D}_{train}) + 4(|\widetilde{q}_{1-\alpha}-{q}_{1-\alpha}|)).
		\end{equation*}

\end{example}

\begin{remark}
In the absence of covariates and with the Fréchet mean $m$ known, the rate of our estimation uniquely depends on the speed at which we estimate the quantiles of pseudoresiduals $r = d(Y, m(X))$. In this particular homoscedastic case, our bounds extend the optimality results established in \cite{gyofi2020nearest} for the univariate case and empirical distribution, yielding a rate of $O(\sqrt{\ln n /n})$.  In the case where the Fréchet mean is not known, but we can estimate, it  for example at the parametric rate $O(1 /\sqrt{n})$, our algorithm can exhibit fast rates.

\end{remark}

  
  \begin{remark}
\noindent More specific non-asymptotic results can be derived for homoscedastic cases and univariate responses, as demonstrated in Corollary 1 of \cite{foygel2021limits}, using the conformal model introduced in \cite{lei2018distribution}. In their work, the authors impose more stringent assumptions to bound the Lebesgue measure, including requirements such as the existence of a density function for the random error and symmetrical conditions, alongside stability criteria for the underlying mean regression functions.

\noindent In contrast, we do not aim to establish convergence under symmetrical conditions or assume the existence of a density function for the random error. While symmetrical conditions could be relevant in practical applications for obtaining interpretable and valuable regions, they are not imperative from a theoretical perspective within our framework. Moreover, assuming the existence of a density function for the random error is a very strong requirement in many real-world applications.
	\end{remark}

	\begin{remark}
In this research, we primarily focus on utilizing our methodology in the setting where $\mathcal{X}=\mathbb{R}^{p}$ is a finite-dimensional Euclidean space. However, it is worth noting that our methodology has the potential to be applied when $\mathcal{X}=\mathbb{H}$, where $\mathbb{H}$ is an arbitrary separable metric space with finite diameter under minimal theoretical conditions or a Hilbert space as the case of predictors that lie in a Reproducing Kernel Hilbert Spaces (see for example this representative regression algorithm \cite{bhattacharjee2023nonlinear}). For more details about the statistical models that accommodate such predictors see \cite{hanneke2022universally}
	\end{remark}

	\subsection{Heterocedastic case}\label{sec:hetero}

	\begin{algorithm}[ht!]
 \caption{Uncertainty quantification algorithm for heteroscedastic setup} \begin{enumerate}
 \item Estimate the function \(\widetilde{m}(\cdot)\) using the random sample \(\mathcal{D}_{\text{train}}\).
		
		\item For a given new point \(x \in \mathcal{X}\), let \(X_{(n_2,1)}(x), X_{(n_2,2)}(x), \dots, X_{(n_2,k)}(x)\) denote the \(k\)-nearest neighbors of \(x\) in increasing order based on the  distance \(d_{2}\). For each data point \(X_i\) in the test dataset \(\mathcal{D}_{\text{test}}\), compute the predicted value \(\widetilde{m}(X_i)\). Then, define the pseudo-residual \(\widetilde{r}_i\) as the distance between the true response value \(Y_i\) and the predicted value \(\widetilde{m}(X_i)\), i.e., \(\widetilde{r}_i = d_{2}(Y_i, \widetilde{m}(X_i))\). Denote the pseudo-residuals associated with the \(i\)-th ordered observation as \(\widetilde{r}_{(n_2,i)}(x)=d_{2}( Y_{(n_2,i)}, \widetilde{m}(x) )\).
		
		\item Estimate the empirical conditional distribution of order $1-\alpha$ by \(\widetilde{G}_{k}^{*}(x,t) = \frac{1}{k} \sum_{i=1}^{k} \mathbb{I}\{\widetilde{r}_{(n_2,i)}(x) \leq t\}\) and denote by \(\widetilde{q}_{1-\alpha}(x)\) the empirical quantile of the empirical distribution \(\widetilde{G}_{k}^{*}(x,t)\).
		
		\item Return the estimated prediction region as $\widetilde{C}^{\alpha}(x):=\widetilde{C}^{\alpha,k}(x) = \mathcal{B}(\widetilde{m}(x), \widetilde{q}_{1-\alpha}(x))$
	\end{enumerate}\label{algor:hete}	  
\end{algorithm}

\noindent	The homoscedastic assumption can be too restrictive in metric spaces that lack a linear structure, making a local approach more appropriate. In this paper, we adopt the $k$-nearest neighbors (kNN) regression algorithm \cite{fix1951discriminatory, biau2015lectures} for this purpose, as it is both computationally efficient and easy to implement.
	
\noindent	In heteroscedastic methods, given a point \(x\), we repeat the same steps as in the homoscedastic algorithm, but the radius estimation only considers data points in a neighborhood of \(x\), denoted as \(N_{k}(x) = \{i \in [S_2]: d_{2}(X_i, x) \leq d_{2(n_2,k)}(x)\}\). Here, \(d_{2(n_2,k)}(x)\) represents the distance between \(x\) and its \(k\)-th nearest neighbor in increasing order in  the test dataset \(\mathcal{D}_{\text{test}}\) using the  distance \(d_{2}\). Algorithm \ref{algor:hete} provides a detailed outline of this procedure.

\begin{remark}    
Although Algorithm \ref{algor:hete} can be consistent, it lacks finite sample guarantees. To address this limitation and propose a heteroscedastic algorithm with non-asymptotic performance guarantees, we suggest extending the conformal quantile methodology introduced by \cite{romano2019conformalized} to our context. Specifically, we aim to achieve the marginal property $\mathbb{P}(Y \in \widetilde{C}^{\alpha}(X;\mathcal{D}_{n})) \geq 1 - \alpha$. The core steps are outlined below:

\begin{enumerate}
    \item \textbf{Estimation of $m(\cdot)$:} Utilize the dataset $\mathcal{D}_{train}$ to estimate the regression function $m(\cdot)$.
    
    \item \textbf{Estimation of $r(\cdot)$:} Employ the random sample $\mathcal{D}_{test}$ and the pairs $(X_{i}, d_{2}(Y_{i}, \widetilde{m}(X_{i})))$ for each $i \in [S_2]$ to estimate the radius function $r(\cdot)$.
    
    \item \textbf{Calibration of $\widetilde{r}(x)$:} In an additional data-split $\mathcal{D}_{test2}$, we calibrate the radius $\widetilde{r}(x)$ to obtain the non-asymptotic guarantees using the scores $S_{i} = d_{2}(Y_{i}, \widetilde{m}(X_{i})) - \widetilde{r}(X_{i})$. The empirical quantile at level $\alpha$, denoted by $\widetilde{w}_{\alpha}$, is then used to define the final radius function as $\widehat{r}(x) = \widetilde{r}(x) + \widetilde{w}_{\alpha}$.
\end{enumerate}
	\end{remark}

	\subsubsection{Statistical Theory}
	
\noindent	The following technical assumptions are introduced to guarantee that the proposed heteroscedastic $k$NN-algorithm is asymptotically consistent.

\begin{assumption}\label{assu}
\begin{enumerate}
        \item $n_2 \to \infty$.
        \item $\widetilde{m}$ is a consistent estimator in the sense that $\mathbb{E}(d_{2}(\widetilde{m}(X), m(X))|\mathcal{D}_{\text{train}}) \to 0$  in probability as $n_1 \to \infty$.
        \item As $n_{2}\to \infty, k \to \infty$, and $\frac{k}{n_{2}} \to 0$.
       \item Except in a set of probability zero, for all \(x \in \mathcal{X}\), the pointwise population quantile \(q_{1-\alpha}(x) = \inf\{t \in \mathbb{R} : G(t , x) = \mathbb{P}(d_{2}(Y, m(x)) \leq t \mid X=x) = 1-\alpha\}\) exists uniquely and  is a continuity point of the function \(G(\cdot, x)\).
\item Except in a set of probability zero, for all \(x \in \mathcal{X}\), \(\mathbb{E}(d_{2}(Y, m(x))\mid X=x) < \infty\).
    \end{enumerate}
\end{assumption}



	\begin{theorem}\label{theorem:hetero}
		Under Assumption \ref{assu}, the uncertainty quantification estimator \(\widetilde{C}^{\alpha,k}(\cdot)\) obtained using Algorithm \ref{algor:hete} satisfies:
		\begin{equation*}
		\lim_{n\to \infty} \mathbb{E}(\varepsilon(\widetilde{C},X)|\mathcal{D}_{n}) \to 0. 
		\end{equation*}
	\end{theorem}

\begin{proof}
The proof is given in the Appendix.
\end{proof}

\begin{remark}
    We can generalize Proposition \ref{the:minimax1} using Assumption \ref{assu} for the kNN algorithm (and introducing another modified Lipchitz condition for the radius estimation $r=d_{2}(Y,m(X))$, see \cite{gyofi2020nearest}):
    \begin{equation*}
        \mathbb{E}(\varepsilon(\widetilde{C}, X) \mid \mathcal{D}_{n}) \leq C (\mathbb{E}(d_{2}(\widetilde{m}(X), m(X)) \mid \mathcal{D}_{\text{train}}) + 4\left( \mathbb{E}( |\widetilde{q}_{1-\alpha}(X) - {q}_{1-\alpha}(X)|)\right)).
    \end{equation*}
    \noindent

\noindent In the heterocedastic case, we must consider explicitly the impact of the conditional quantile $q_{1-\alpha}(x)$ for each $x \in \mathcal{X}$ in the rate. In a similar context, Györfi et al. \cite{gyofi2020nearest} propose the optimal rate for prediction intervals, which coincides when the conditional regression function \(m\) is known and is given by \(O\left(\frac{\ln n}{n^{\frac{1}{p+2}}}\right)\) for \(p \geq 1\), where $p$ denotes the dimension of the predictor $X$. It is important to note that this rate scales with the dimension of the covariate space \(p\), unlike the homoscedastic or non-covariate settings.

\end{remark}
	
	\subsection{Global Fréchet regression model}\label{sec:global}
\noindent	In this section, we introduce some background on regression modeling in metric spaces, which is essential for implementing our uncertainty quantification algorithms in practice. After motivating the need for models that handle random responses in metric spaces, particularly in clinical applications, the primary objective of this section is to introduce the global Fréchet regression model. This regression model in metric spaces will be employed across various analytical tasks.


\noindent The goal is to model the regression function
\[
	m(x)= \argminB_{y\in \mathcal{Y}} \mathbb{E}(d_{1}^{2}(Y,y)\mid X=x)
\]
so that its estimation becomes feasible.

\noindent In this section, we focus on a specific regression model designed for metric-space valued outcomes, known as the \emph{global Fréchet regression model} \cite{petersen2019frechet}.  

\noindent  
The global Fréchet regression model provides a natural extension of the standard linear regression model from Euclidean spaces to metric spaces. 

\begin{figure}[ht!]
	\centering
	\includegraphics[width=0.9\linewidth]{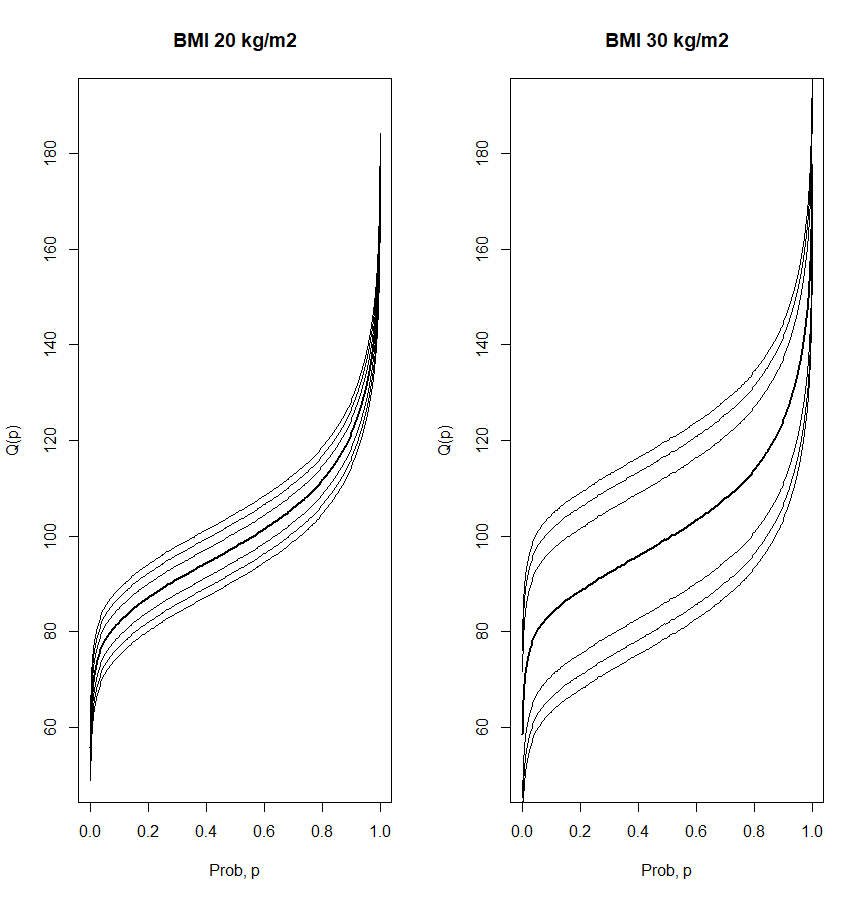}
	\caption{Uncertainty quantification from the quantile glucose representations for two non-diabetic patients with the same clinical characteristics, with the exception of body mass index (BMI). In patient (i), \(BMI = 20 \, \text{kg/m}^2\), and in patient (ii), \(BMI = 30 \, \text{kg/m}^2\) (more uncertainty).}
	\label{fig:bmiintro}
\end{figure}

\noindent Our interest in this model stems from its potential applications in the field of medical data analysis. As medical technology continues to advance, we expect to have access to increasing amounts of complex data about patients' health conditions.

\noindent To illustrate the significance of uncertainty quantification in this context, we turn our attention to the continuous glucose monitoring technology.  Consider two non-diabetic individuals who share identical clinical characteristics, except for their body mass index. Our primary focus is to predict their glucose quantile functions. As depicted in Figure \ref{fig:bmiintro}, both individuals exhibit the same conditional mean function (indicated by the black line). However, it is important to note that the uncertainty in the prediction curves increases with higher body mass.

\noindent In practical terms, if we were to use the conditional mean results to modify the diet based on changes in the glucose profiles of overweight patients, we would encounter technical difficulties. The patient responses in this subgroup turn out to be heterogeneous and individualized responses, as evidenced by the varying levels of uncertainty in the predictions. As a consequence, relying solely on prescriptive data-driven algorithms based on conditional mean regression function would prove inadequate for optimizing diet selection among these patients.





\noindent To effectively utilize uncertainty quantification algorithms, we 
consider the global Fréchet model, a generalization of linear regression to metric-space valued responses.

\noindent 
In particular, \cite{petersen2019frechet}
propose to model the regression function in the form
\begin{equation*}
m(x)=\underset{y \in \mathbb{R}}{\operatorname{argmin}} \mathbb{E}\left[\omega(x, X) d(Y,y)^{2}\right]~,
\end{equation*}
where the weight function $\omega$ is defined by
\begin{equation*}
\omega(x,z)=1+(z-\mu)^{\top} \Sigma^{-1}(x-\mu) 
\end{equation*}
with $\mu=\mathbb{E}(X)$ and $\Sigma=\operatorname{Cov}(X)$.
As \cite{petersen2019frechet} show, in the case of $\mathcal{Y}=\mathbb{R}$,
this model reduces to standard linear regression.





\noindent While in our applications we primarily focus on developing estimators in the global Fréchet regression model, it is important to emphasize that our uncertainty quantification algorithm can be applied to any regression technique, including random forests in metric spaces. It should be noted, however, that if the global Fréchet models do not align with the conditional Fréchet mean \( m \), the resulting estimators will be inconsistent. 


\noindent To estimate  the conditional mean function $m(x)$ under the
 global Fréchet model from a random sample $\mathcal{D}_n=\{(X_i,Y_i)\}_{i=1}^n$, we may solve the counterpart empirical problem

\begin{equation}\label{eqn:linear}
	\widetilde{m}(x)= \argminB_{y \in \mathcal{Y}} \frac{1}{n} \sum_{i=1}^{n} [\omega_{in}(x) d_{1}^{2}(y,Y_i)],
\end{equation}

\noindent where $\omega_{in}(x)= \left[ 1+(X_{i}-\overline{X})\widetilde{\Sigma}^{-1}(x-\overline{X})\right],$ with $\overline{X}= \frac{1}{n} \sum_{i=1}^{n} X_i$, and $\widetilde{\Sigma}= \frac{1}{n-1} \sum_{i=1}^{n} (X_i-\overline{X})(X_i-\overline{X})^{\top}$



	
    \subsection{Testing Homoscedastic and Heteroscedastic Nature of the Uncertainity Quantification Model}

In this section, we introduce rigorous criteria for testing whether the distribution $(X, Y)$ is homoscedastic.

		\noindent From a practical standpoint, model users are faced with a choice between employing the homoscedastic or heteroscedastic algorithm. In the former case, finite sample guarantees are available, whereas in the latter case, only consistency guarantees exist. If the underlying distributional model is not homoscedastic and we use such an algorithm, the estimated conditional predictions regions would be inconsistent. From a mathematical statistical perspective, the appropriate criterion for deciding which model to choose is to test the null hypothesis that, for all \(x \in \mathcal{X}\), and $r\geq 0$:
\begin{equation}
\mathbb{P}(Y \in \mathcal{B}(m(x),r) \mid X=x)
\end{equation}
\noindent is independent of  $x$.

\noindent Alternatively, we can define the problem as to test that the random variable of the pseudo-residuals \(r = d_{2}(Y, m(X))\) is independent of \(X\), i.e., $H_{0}: r \perp X$ versus $H_{a}: r \not\!\perp X$.

\noindent To achieve this, we first estimate the regression function \(m\) using the random elements of \(\mathcal{D}_{\text{train}}\), obtaining an estimator \(\widetilde{m}\). Subsequently, we construct a hypothesis test with the random sample \(\mathcal{D}_{\text{test}} = \{X_i, d_{2}(Y_i, \widetilde{m}(X_i))\}_{i \in [S_{2}]}\). Due to the structural hypothesis, as \(\widetilde{m}\) is a consistent estimator, meaning that \(\mathbb{E}(d_{2}(\widetilde{m}(X), m(X))|\mathcal{D}_{\text{train}}) \to 0\) in probability as \(|\mathcal{D}_{\text{train}}| \to \infty\), it is expected,  that hypothesis testing for homoscedasticity is consistent.

\noindent In this article, we propose employing an existing multivariate independence testing methodology—distance covariance or distance correlation (the standardized version) \cite{szekely2007measuring, corrdepend} — as a formal criterion for testing the homocedasticity of $(X,Y)$ with respect to the regression function $m$. This builds upon the original hypothesis testing framework of homoscedasticity from \cite{goldfeld1965some} for univariate responses. Our approach is methodologically valid for testing the homocedastic nature of the pair $(X,Y)$ when the predictor $X$ takes values in separable Hilbert spaces \cite{lyons2013distance}.

\noindent Let us start with the definition of the sample distance covariance for the Euclidean distance  \(\|\cdot-\cdot\|\).

\noindent For all \(i \in [S_{2}]\), consider the bivariate random element \((X_i, d_{2}(Y_i, \widetilde{m}(X_i)))\). 

\begin{definition}
     The squared sample distance covariance (a scalar) is  the arithmetic average of the products \(A_{j,k} \cdot B_{j,k}\):
\begin{align*}
\widetilde{\operatorname{dCov}}^{2}(X,r) = \frac{1}{n_{2}^{2}}\sum_{j\in  [S_{2}]}\sum_{k\in  [S_{2}]} A_{j,k} \cdot B_{j,k}, 
\end{align*}

\noindent where
\begin{align*}
A_{j,k} &= a_{j,k} - \overline{a}_{j\cdot} - \overline{a}_{\cdot k} + \overline{a}_{\cdot\cdot} &j, k \in [S_{2}],\\
B_{j,k} &= b_{j,k} - \overline{b}_{j\cdot} - \overline{b}_{\cdot k} + \overline{b}_{\cdot\cdot}, &j, k \in [S_{2}]
\end{align*}

\noindent with

\begin{align*}
a_{j,k} &= \|X_{j} - X_{k}\|, &j, k \in [S_{2}],\\
b_{j,k} &= \|d_{2}(Y_j, \widetilde{m}(X_j)) - d_{2}(Y_k, \widetilde{m}(X_k))\|, &j, k \in [S_{2}],
\end{align*}

\noindent  and \(\overline{a}_{j\cdot}\) is the \(j\)-th row mean, \(\overline{a}_{\cdot k}\) is the \(k\)-th column mean, and \(\overline{a}_{\cdot\cdot}\) is the grand mean of the distance matrix of the \(X\) sample. The notation is similar for the \(b\) values.  \(\|\cdot\|\) denotes the Euclidean norm. 
\end{definition}

\noindent 
The following consistency result enables us to develop a consistent hypothesis test to determine whether the random pair $(X, Y)$ exhibits homoscedasticity with respect to the function $m$.

 \begin{proposition}
Suppose that $\mathbb{E}(\|X\|) < \infty,$  $\mathbb{E}(r) < \infty$, and \(\mathbb{E}(d_{2}(\widetilde{m}(X), m(X))|\mathcal{D}_{\text{train}}) \to 0\) in probability as \(|\mathcal{D}_{\text{train}}| \to \infty\). Then

\begin{equation}
 \lim_{n_{1} \to \infty}     \lim_{n_{2}\to \infty} \widetilde{\operatorname{dCov}}^{2}(X,r)= \operatorname{dCov}^{2}(X,r),   
\end{equation}

\noindent where $\operatorname{dCov}(X,r)$ is the population definition of the  distance covariance introduced in  \cite{szekely2007measuring}.

\end{proposition}
     
\begin{proof}
Theorem $2$ from \cite{szekely2007measuring} using the hypothesis \(\mathbb{E}(d_{2}(\widetilde{m}(X), m(X))|\mathcal{D}_{\text{train}}) \to 0\) in probability as \(|\mathcal{D}_{\text{train}}| \to \infty\).
\end{proof} 
 
\noindent Then,  the test statistic \(T_{n_2} = n_2 \widetilde{\operatorname{dCov}}^{2}(X,r)\) determines a statistical tests asymptotically valid to any deviation (alternative) from the null hyphotesis. For an implementation, see the \texttt{dcov.test} function in the \texttt{energy} package for R.

\noindent From a practical perspective, we provide several alternatives to calibrate the test statistics under the null hyphotesis and derive $p$-values, including permutation methods, resampling methods based on Efron's Naive bootstrap, or the Wild bootstrap \cite{szekely2007measuring, matabuena2022kernel, doi:10.1080/00031305.2023.2200512}. Each of these calibration options offers valuable means to fine tune the statistical tests, for diverse analysis requirements.

	\subsection{New model constructions:   variable selection methods for regression models in metric spaces}\label{sec:sel}

\begin{figure}[H]
	\centering
	\includegraphics[width=0.9\linewidth]{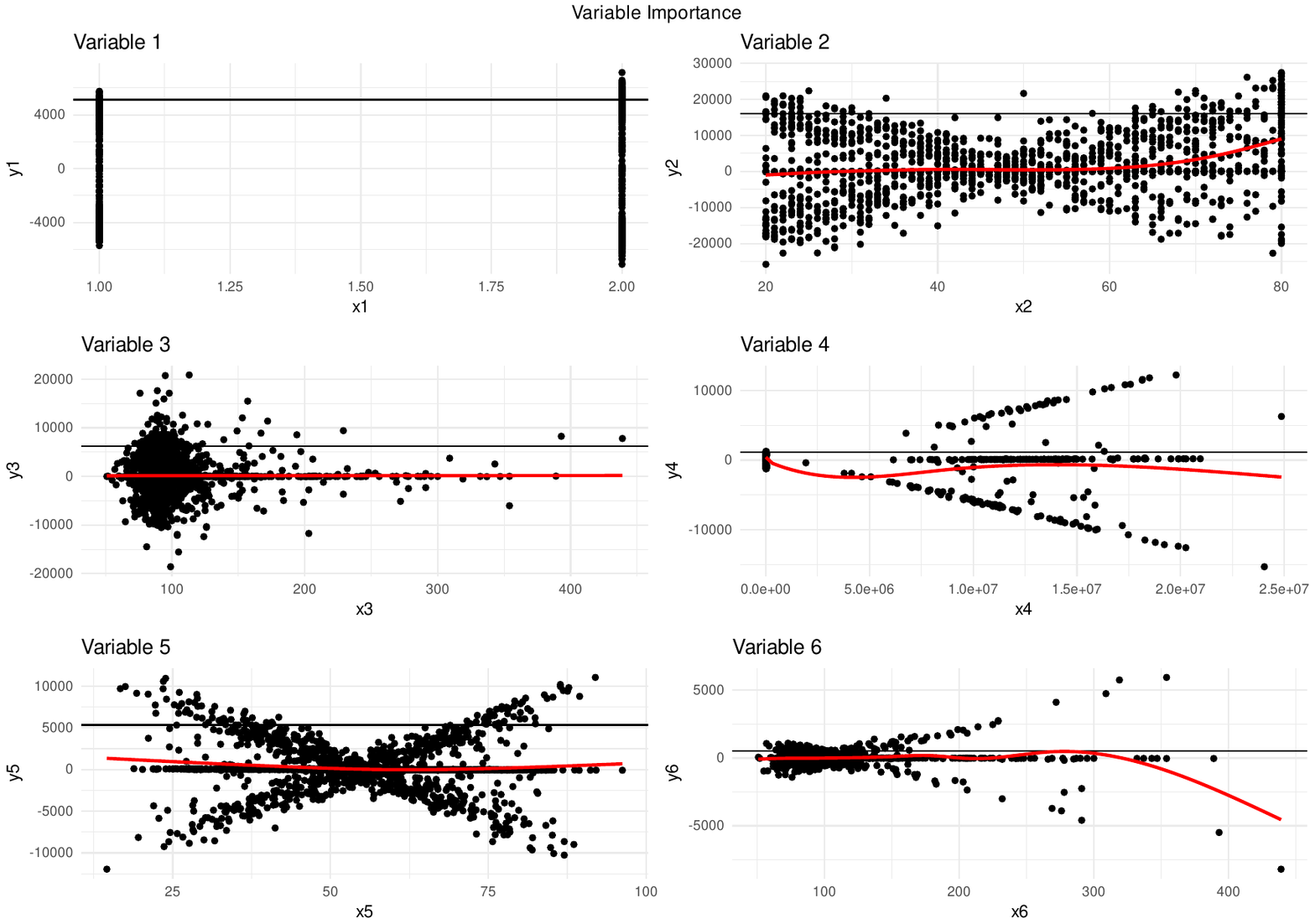}
	\caption{An example of the new local variable  model, including six predictors. The first predictor is a binary variable representing the individual's sex, and the remaining variables are continuous biomarkers. Based on the results obtained, all variables except for variables 3 and 5 were found to be important globally in this example.}
	\label{fig:sel}
\end{figure}
  
\noindent In this section, we present a  variable selection model that can be applied to any arbitrary regression models operating within metric spaces. Before we begin, it is important to clarify that we utilize uncertainty quantification algorithms applied over a transformed random variable \(W\) on the real line, then we can apply over traditional conformal inference techniques. The critical  modeling step is to combine the diferent steps of our uncertainty algorithms to create a variable selection algorithm using scalar methods.

\noindent  For simplicity, we consider that the underlying regression function \( m \) is defined through the global Fr\'echet Regression model

\begin{equation}\label{eqn:global}
	m(x)=  \argminB_{y \in \mathcal{Y}} \mathbb{E}\left[\omega\left(x,X\right) d_{1}^{2}(Y,y) \right],
\end{equation}

\noindent where $\Sigma$ is the covariance matrix of $X$, $\mu= \mathbb{E}(X)$, $\omega(x,X)=  \left[1+(X-\mu)\Sigma^{-1}(x-\mu)\right]$.

\noindent Let \(m(\cdot)\) denote the regression function with the $p$- covariates and \(m_{-j}(\cdot)\) denote the regression function obtained by excluding the \(j\)th variable from the model. To evaluate the relevance of the \(j\)th predictor, we introduce the random variable \(W(X,Y)= d_{1}^{2}(Y,m_{-j}(X))-d_{1}^{2}(Y,m(X))\) such that $\mathbb{E}(W(X,Y))\leq 0$. From a theoretical point of view, under the null hypothesis, \(H_0: j \in \{1,\dots,p\}\) is an irrelevant predictor, the random variable \(W(X,Y)\) is equal to zero with probability one. In order to test this null hypothesis, we construct a univariate prediction interval following the methodology established in Sections \ref{sec:homo} and \ref{sec:hetero} respectively (for $\mathcal{Y}=\mathbb{R}$), denoted as \(\widetilde{C}^{\alpha}(X)\), where \(\mathbb{P}(W(X,Y) \in \widetilde{C}^{\alpha}(X)) \geq 1-\alpha\), and we check if this interval contains only positive values. This hypothesis can be tested locally and globally for a pre-specified confidence level \(\alpha \in [0,1]\).

\noindent The new variable selection method offers several significant advantages over existing variable selection methods in metric spaces \cite{tucker2021variable}. Firstly, it provides a local measure of the importance of the $j$th variable, as demonstrated in Figure \ref{fig:sel}. Unlike its competitors, our method can be applied in more general settings beyond linear models, including additive or semiparametric models \cite{lin2023additive}, and it effectively handles categorical predictors (with two or more levels as group-Lasso models \cite{wang2008note}). Notably, the existing variable selection methods for metric space responses \cite{tucker2021variable} inform about the global importance of the $j$th predictor, while our approach is local, being more flexible and informative about the contribution of a predictor in a specific region of the prediction space.

\noindent Secondly, our method does not necessitate variable standardization, making it more practical and straightforward to use. Additionally, it provides a \(p\)-value that reflects the strength of each variable, enabling an assessment of their significance.

\noindent To obtain a global \(p\)-value indicating the significance of each variable, we introduce an additional notation. We define the functions \(\widetilde{W}_{j}(x,y)= d_{1}^{2}(y,\widetilde{m}_{-j}(x))-d^{2}_{1}(y,\widetilde{m}(x))\) and consider the median or integrated median computed from the random elements of the sample \(\mathcal{D}_{\text{test}}\). For instance, we calculate \(\widetilde{w}_{j}= \frac{1}{|\mathcal{D}_{\text{test}}|} \sum_{i\in \mathcal{D}_{\text{test}}} \widetilde{W}_{j}(X_i,Y_i)\) and employ a \(t\)-test to evaluate the null hypothesis \(H_{0}:w_{j}\leq 0\) against \(H_{a}:w_{j} > 0\). Alternatively, a Wilcoxon test can be used, which tends to be more robust in the presence of outliers. In this paper, we employ a Wilcoxon test to assess the significance of each variable and apply a Bonferroni correction to adjust the confidence level for multiple comparisons. Along this paper, we implement this variable selection procedure with the homoscedastic algorithm defined in Section \ref{sec:homo}.
	
	\section{Simulation study}

	\noindent We conducted a comprehensive simulation study to rigorously evaluate the empirical performance of the proposed uncertainty quantification framework with finite samples. The study encompassed a wide range of scenarios and diverse data structures.

	\subsection{Multidimensional Euclidean data}
	
	\subsubsection{Homoscedastic case}
	
\noindent To evaluate the methods performance under homoscedastic conditions and controlled settings, we conducted a simulation study based on the classical multivariate Gaussian linear regression model. Specifically, we consider a $p$-dimensional Gaussian predictor variable $X = (X_1, \dots, X_p)^\top \sim \mathcal{N}((0, \dots, 0)^\top, \Sigma)$, where $(\Sigma)_{ij} = \rho \in [0,1]$ for all $i\neq j$ and $(\Sigma)_{ii} = 1$ for all $i=1,\dots,p$. Additionally, we consider an $s$-dimensional random response variable $Y = (Y_1, \dots, Y_s)^\top \in \mathbb{R}^{s}$, modeled by the functional relationship:

\begin{equation}
Y = m(X) + \epsilon = X^\top \beta + \epsilon,
\end{equation}

\noindent where $\beta$ is a $p \times s$ matrix of coefficients with all entries equal to 1. The random error $\epsilon = (\epsilon_1, \dots, \epsilon_s)^\top \sim \mathcal{N}(0, \mathrm{diag}(1, \dots, 1))$ is assumed to be independent from $X$ and satisfies $\mathbb{E}(\epsilon|X) = 0$.

\noindent To assess the performance of our methods under various scenarios, we conducted a comprehensive simulation study with multiple cases. We explored different configurations for the values of $s$ and $p$, as well as various sample sizes, confidence levels, and correlation structures for the covariance matrix $\Sigma$, governed by the parameter $r\in [0,1]$. Specifically, we considered $s \in \{1,2,5\}$ and $p \in \{1,2,5,10\}$, while the sample size $n$ varied between $500$ and $5000$ data points. The confidence levels $\alpha$ ranged from $0.05$ to $0.5$, and we set correlation parameters $r$ equal to $0, 0.2, 0.4$, and $0.6$. Here, we only show the results in an intermediate range with $s=2$, $p=0.5$, $r=0.2$, and $n \in \{500,1000,5000\}$.

\noindent To ensure the robustness of our findings, we employed a training set of size $|\mathcal{D}_{train}| = n/2$ and a test set of size $|\mathcal{D}_{test}| = n - |\mathcal{D}_{train}|$. Each scenario was repeated $500$ times, and we evaluated the finite marginal coverage in an additional test set $|\mathcal{D}_{test_{2}}| = 5000$ drawn from the same probability law.

\noindent For estimation purposes, we obtained estimates of the regression function \(m\) using a multivariate response linear regression model with the standard mean square criteria \cite{seber2003linear}. Subsequently, we applied Algorithm \ref{alg:metd1}, as defined in Section \ref{sec:homo}, to generate prediction regions. The results for various confidence levels and sample sizes are illustrated in Figure \ref{fig:simul1}. It is evident that the algorithm consistently maintains the desired marginal coverage, and as the sample size increases, the variability decreases. This observation indicates the good performance of our algorithm across diverse scenarios, in accordance with its non-asymptotic guarantees and theoretical consistency properties.

	\begin{figure}[ht!]
	\centering
	\includegraphics[width=0.9\linewidth]{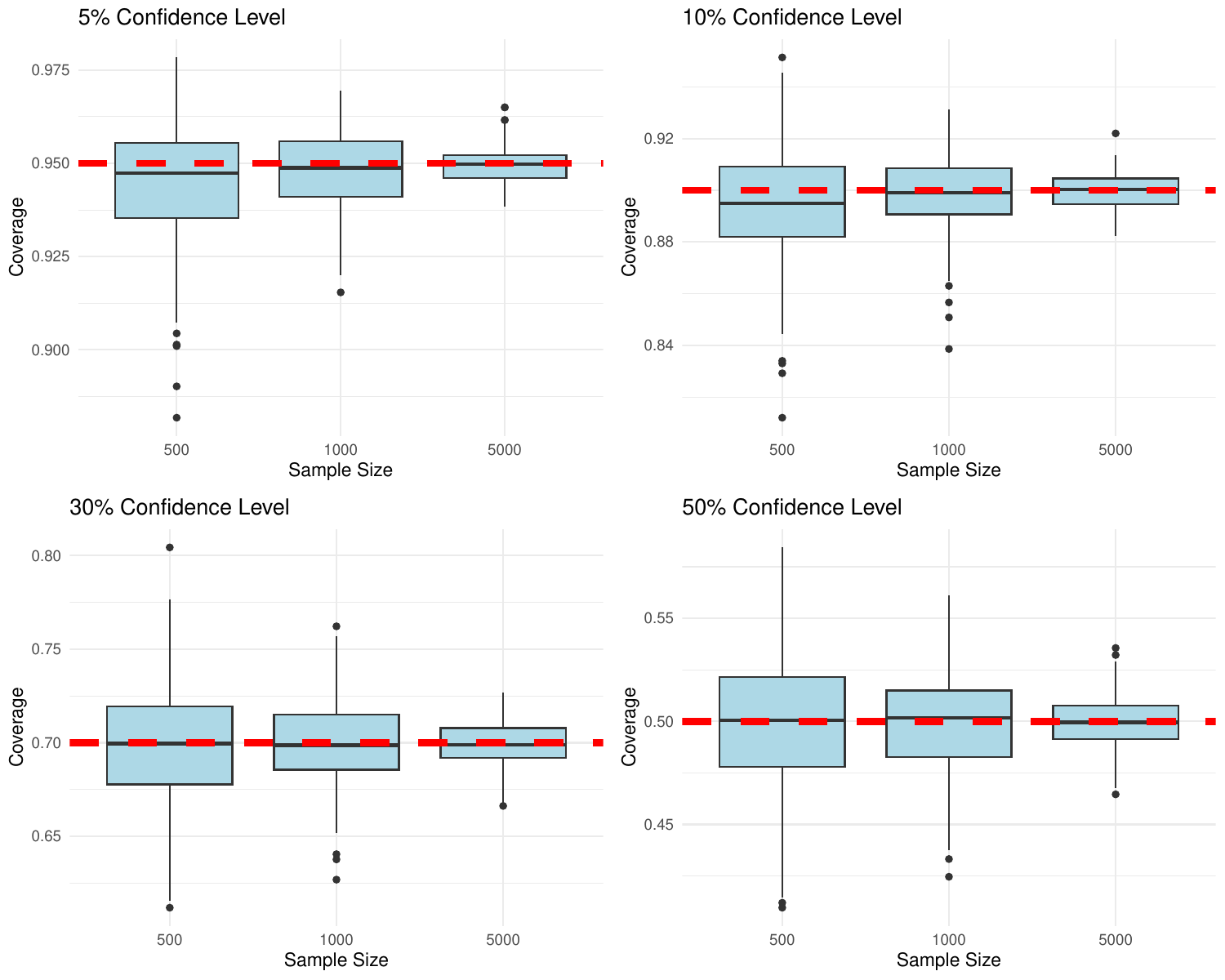}
	\caption{The figure shows the distribution of the marginal coverage $\mathbf{P}(Y\in \widetilde{C}^{\alpha}(X))$ in the homoscedastic case across different sample sizes and confidence levels. The results of our analysis provide insights into how the marginal coverage varies as a function of sample size and confidence level.}
	\label{fig:simul1}
\end{figure}

	\subsubsection{Heterocedastic case}
	
\noindent To account for heteroscedasticity in the simulation scenarios, we modify the previous simulation scheme and incorporate an additional term $\sigma: \mathbb{R}^{p}\to \mathbb{R}$ that represents the heteroscedastic term in the model. Specifically, we set the random error term to be $\sigma(X)\epsilon$, where $\sigma(x) = \norm{x}_2$ is a function that maps each feature vector to a scalar value corresponding to their norm. The overall generative model we consider is:

\begin{equation}
Y = m(X)+\sigma(X)\epsilon = X^{\intercal}\beta + \norm{X}_2\epsilon,
\end{equation}

\noindent To estimate the regression function, we use linear regression without incorporating the heteroscedastic covariance structure in the model estimation. Although this estimator is not the most efficient in the sense of the Gauss-Markov theorem \cite{harville1976extension}, which deals with unbiased and minimum variance characters, the estimator remains consistent.

\noindent In this case, we use the algorithm presented in Section \ref{sec:hetero}, denoted as Algorithm \ref{alg:hete}, and utilize the k-nearest neighbor (kNN) algorithm with varying values of $k$: $k=10, 20, 50,$ and $100$. We evaluate the marginal coverage of the resulting predictions to assess the accuracy of the model. The results of our analysis are presented in Figure \ref{fig:simul2}, following a format similar to the previous example. This figure offers  insights into our methods performance across various values of $k$ and assists in determining the optimal choice of $k$. The selection of this optimal value is guided by specific uncertainty quantification results, which are tailored to the characteristics of the dataset and the coverage achieved during the analysis of the results.

\noindent The analysis of empirical results highlights the sensitivity of the method's performance to the choice of the parameter $k$, especially when dealing with small confidence levels and situations where multiple neighbors are required to approximate extreme quantiles. Across various dimensions of the response variable and predictors, we see that the method maintains stable performance as the sample size increases. However, it is important to acknowledge that in the heteroscedastic case, tuning the $k$ parameter on a case-by-case basis may be necessary. We are well aware, both from empirical evidence regarding the use of k-nearest neighbors (kNN) for estimating conditional means and theoretical results, that the choice of the smoothing parameter $k$ is of utmost importance in achieving estimators with good empirical properties (see \cite{zhang2018novel}). 


\noindent Lastly, we must note that at a median level corresponding to $\alpha$ parameter of $0.5$, the marginal coverage is consistent across different situations examined, demonstrating the robustness of the method in detecting prediction regions that contain half of the sample.

\begin{figure}[H]
  \centering
  \includegraphics[width=0.90\textwidth]{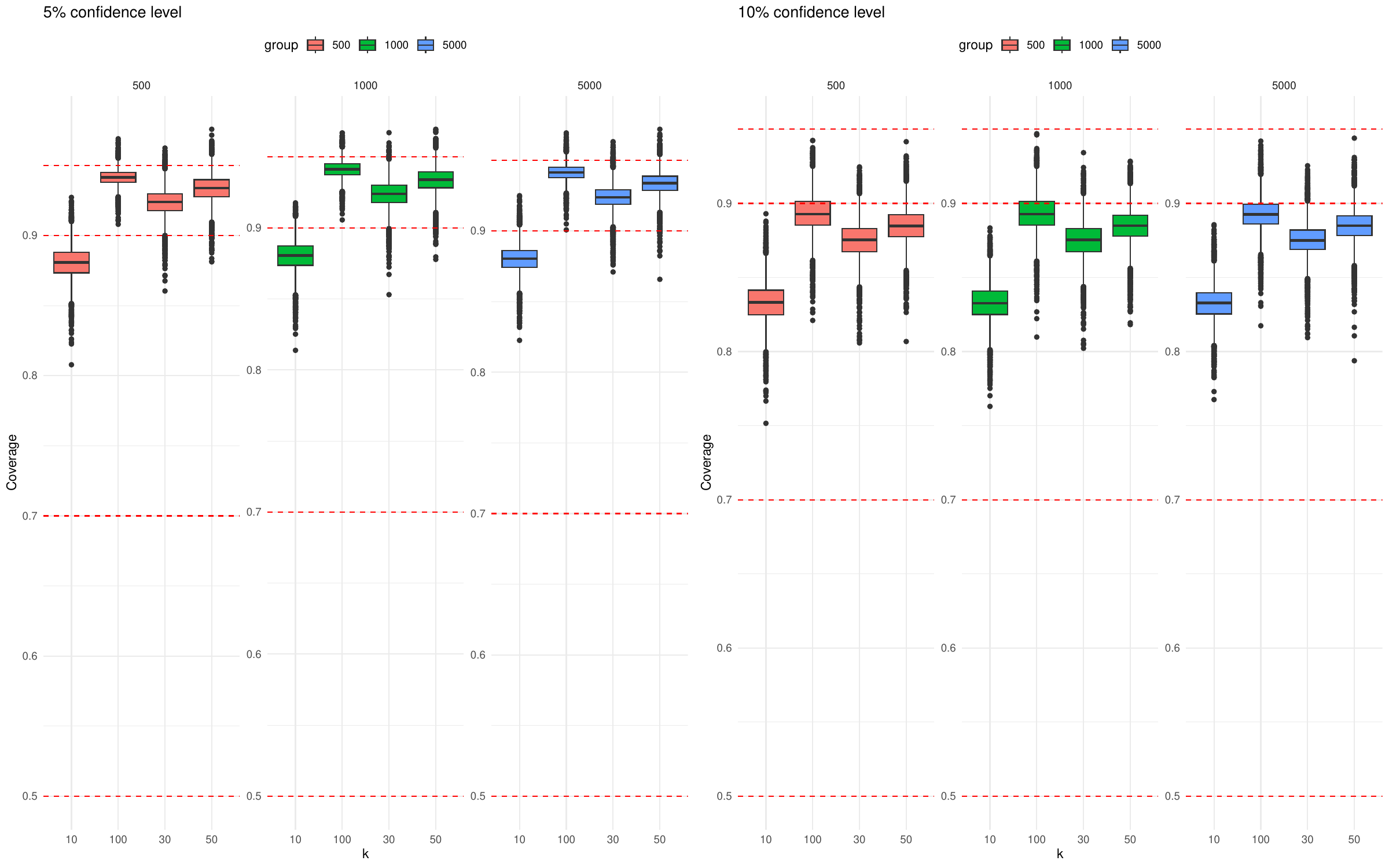}
  \label{fig:combined_boxplots}
  
  \vspace{1em} 
  
  \includegraphics[width=0.90\textwidth]{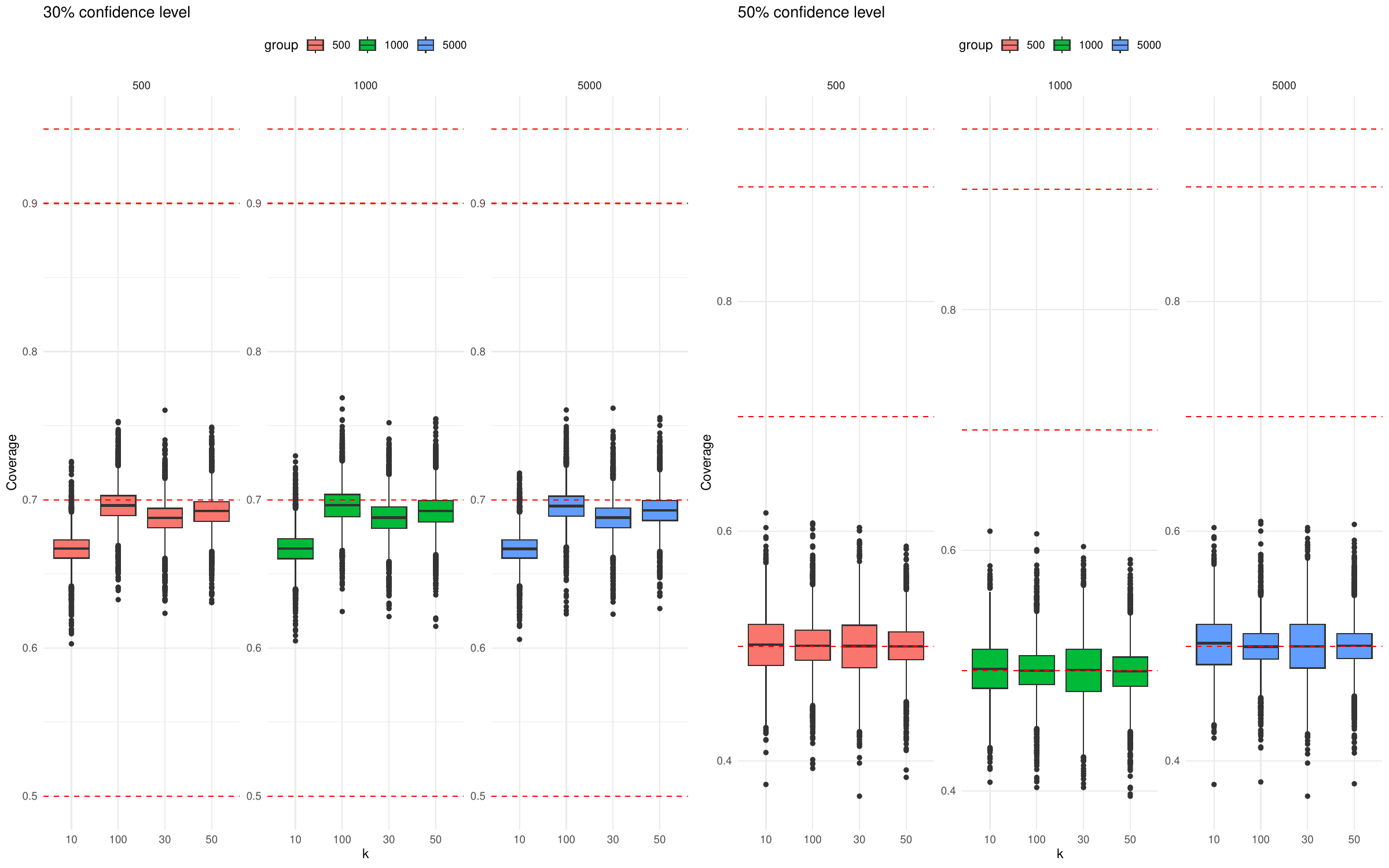}
  \label{fig:individual_boxplots}
  
  \caption{The figure illustrates how the distribution of the marginal coverage $\mathbf{P}(Y\in \widetilde{C}^{\alpha}(X)$ in the heterocedastic case varies across different sample sizes and confidence levels as the number of neighbors changes. Our results provide insights into the behavior of the marginal coverage and the impact of the number of neighbors on the performance of our method.}
  \label{fig:simul2}
\end{figure}

	\subsection{Examples of probability distribution}

	\subsubsection{Uncertainty quantification}
	
	 \noindent Let us revisit Example \ref{ejempldos} from the introduction, , where we develop a model to represent observations drawn from a latent continuous process by utilizing their distribution function. Formally, we define the spaces \(\mathcal{X} = \mathbb{R}^p\) and \(\mathcal{Y} = \mathcal{W}_2(\mathbb{R})\). Here, \(\mathcal{W}_2(\mathbb{R})\) represents the \(2\)-Wasserstein space, equipped with the \(2\)-Wasserstein metric \(d^{2}_{\mathcal{W}_2}(F,G) = \int_{0}^1 (Q_F(t)-Q_G(t))^2 dt\).  \(Q_F\) and \(Q_G\) denote the quantile functions of the distribution functions \(F\) and \(G\), respectively.

\noindent For our specific simulation purposes, we suppose that the underlying generative model is defined:

\[Z(j) = 100 + 5 \cdot \sum_{l=1}^{p} X_l + \epsilon(j), \quad j = 1, \dots, n.\]

\noindent Let

\begin{equation}
\widetilde{F}(\rho) = \frac{1}{n} \sum_{j=1}^{n} \mathbb{I}\{Z(j) \leq \rho\}.
\end{equation}

\noindent  We set \(n = 300\). The remaining parameters are established as in the previous simulations related to the multivariate Euclidean data, with \(p=5\) and \(n_{i}\) ranging from \(500\) to \(5000\). We use a training set of size \(|\mathcal{D}_{\text{train}}| = n/2\) and a test set of size \(|\mathcal{D}_{\text{test}}| = n - |\mathcal{D}_{\text{train}}|\), and repeat each scenario \(500\) times, evaluating the finite marginal coverage at additional test points \(|\mathcal{D}_{\text{test2}}| = 5000\) using the same probability law as the pair \((X,Z)\).

\noindent For estimation purposes, we utilize the Global-Fréchet model equipped with the \(2\)-Wasserstein distance to predict the quantile response variable. We consider the Euclidean distance, denoted as \(d_2\), between the predicted quantiles and the observed estimated quantile function. The confidence level \(\alpha\) is fixed with values varying from \(0.05\) to \(0.5\). It is important to note that in our example, we use estimated quantile functions from a time series as the response variable, introducing additional sources of uncertainty and making it challenging to obtain exact marginal control. Despite this, the underlying model follows a homoscedastic definition introduced in Section \ref{sec:def}. The results, presented in figure \ref{fig:simul3}, confirm this hypothesis and indicate that the obtained coverage is anti-conservative (below the nominal value) but very close to the real nominal values. Again, we  note that at a median level corresponding to \(\alpha=0.5\), the marginal coverage is consistent across different situations examined, demonstrating the robustness of the method in detecting prediction regions that contain half of the sample.
 

	\begin{figure}[H]
		\centering
		\includegraphics[width=0.9\linewidth]{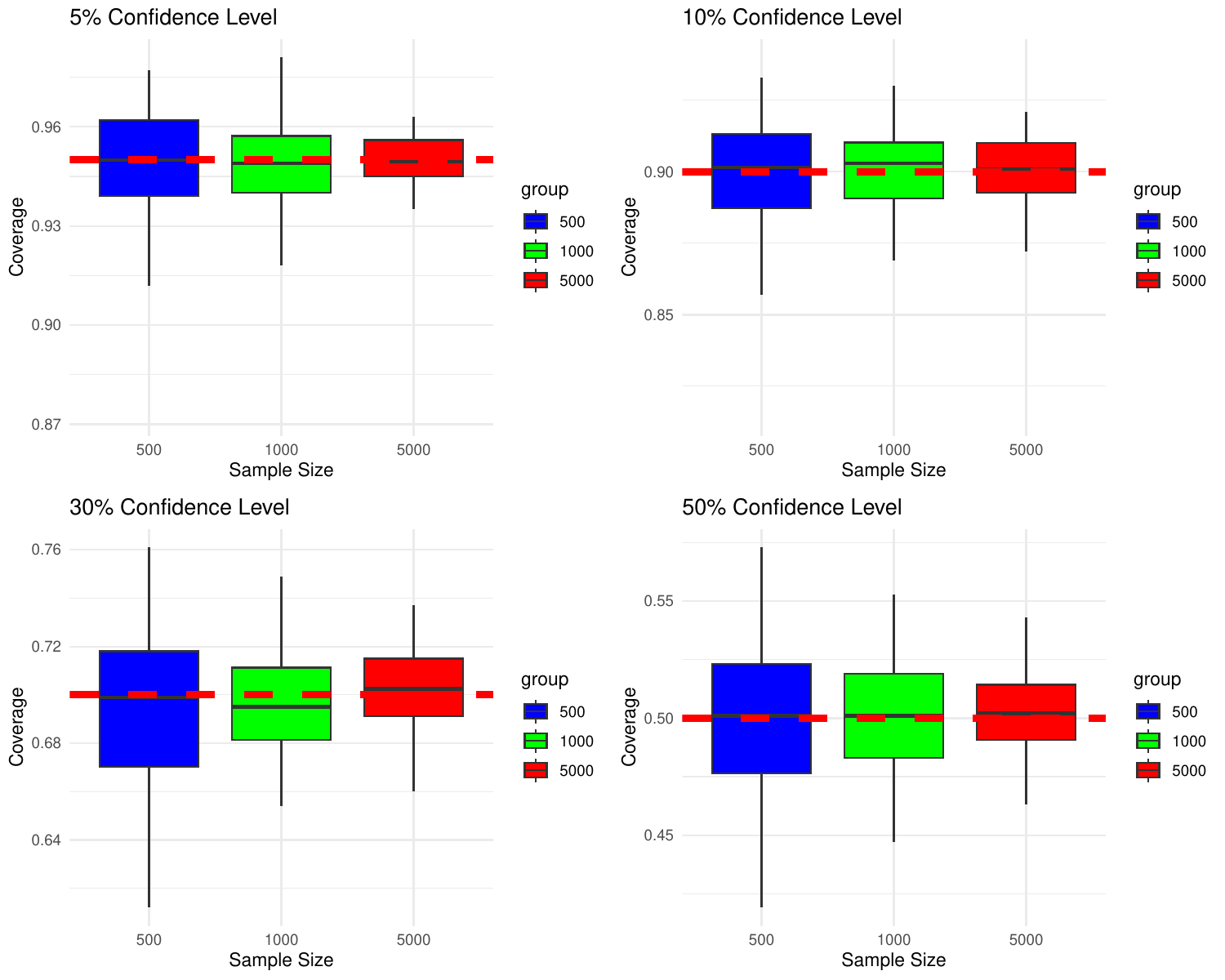}
		\caption{The figure presents the distribution of the marginal coverage across different confidence levels and sample sizes in the distributional data analysis example. The analysis utilized estimated quantile functions from a time series as the response variable, introducing additional sources of variability and making it challenging to obtain exact marginal control, despite the underlying regression model being homoscedastic. Nonetheless, the results indicate that the method is effective in obtaining conservative coverage levels that are very close to the true values.}
		\label{fig:simul3}
	\end{figure}

	\subsubsection{Variable selection}
\noindent Continuing with the previous example, we now consider the case where the number of variables that appear in the  generative model is equal to $p_{\text{true}}=1$ (the first covariate for simplicity) across all scenarios. For each $j=1,\dots,n$, the specific generative model is given by

\[
Z(j) = 100 + 5 \cdot X_{1} + \epsilon(j).
\]

\noindent To identify the most relevant variable for the model, we employ the variable selection strategies outlined in Section \ref{sec:sel}, with a confidence level of $\alpha = 0.05$. For this purpose, we utilize the Global-Fréchet regression model as a base regression model to predict the empirical quantile function.

\noindent An important point in the variable selection modeling strategy, highlighted in Section \ref{sec:sel}, is that we are testing simultaneously $p$-hypotheses, and we must address the statistical problems of multiple comparisons. In this paper, we implement the Bonferroni method for multiple comparisons. Table \ref{table:resultssel} displays the results across different sample sizes, demonstrating the statistical consistency of the variable selection approach as the sample size increases in this example.

\begin{table}[htbp]
    \centering
    \resizebox{\textwidth}{!}{%
        \begin{tabular}{@{}lccc@{}}
            \toprule
            & \textbf{Detect True Variables (\%)} & \textbf{\begin{tabular}[c]{@{}c@{}}One Variable\\ False Positive (\%)\end{tabular}} & \textbf{\begin{tabular}[c]{@{}c@{}}Two or More Variables\\ False Positive (\%)\end{tabular}} \\
            \midrule
            \textbf{$n=500$}  & 100 & 68 & 26 \\
            \textbf{$n=1000$} & 100 & 29 & 4  \\
            \textbf{$n=5000$} & 100 & 3  & 1  \\
            \bottomrule
        \end{tabular}%
    }
    \caption{Results of Simulations for Variable Selection Approach}
    \label{table:resultssel}
\end{table}
	
	\section{Data analysis examples}
	
In this section, we demonstrate the wide-ranging utility and adaptability of our algorithm across different data structures and statistical modeling scenarios through a series of real-world examples. Our emphasis is on showcasing the algorithm's applicability in modern medical contexts.

\noindent
The initial example explores situations lacking predictors, utilizing insights from shape analysis to characterize Parkinson's disease using digital biomarkers. This contrasts with the next four examples, which examine pairs \((X, Y) \in \mathcal{X} \times \mathcal{Y}\), offering a broader perspective on the algorithm's performance in regression analyses. 

\subsection{Shape Analysis}

\subsubsection{Motivation}

Shape analysis is a research area that requires novel statistical methods in metric spaces \cite{dryden2016statistical,steyer2023elastic} and has broad applications in the biomedical sciences. This approach is particularly useful when analyzing random objects based on their geometrical properties. Here, we are motivated to analyze the shapes of handwritten samples from patients with a neurological disorder versus those from a control group of healthy individuals. The scientific problem we aim to address in this section is to define uncertainty prediction regions for healthy individuals at some confidence level $\alpha \in (0,1)$, using handwriting trajectories. Assuming that individuals with the neurological disorder exhibit behaviors that deviate significantly from those expected of the control group in their writing shapes. We use these prediction regions to identify patients who differ markedly from the control group and to detect and characterizate outlier behavior.

\subsubsection{Data Description}

This study employs data from \cite{isenkul2014improved}, comprising $n_{1}=15$ control subjects and $n_{2}=30$ individuals diagnosed with Parkinson's disease, to analyze handwriting shapes captured through digital tests. Our objective is to establish expected predictions regions at different confidence levels $\alpha$ based on the control group's data. For this purpose, first we estimate the empirical  Fréchet mean using Euclidean distance $d^{2}_{1}(\cdot,\cdot)= \|\cdot-\cdot \|^{2}$. Subsequently, we adopt $d_{2}(\cdot,\cdot)= \|\cdot-\cdot \|_{\infty}$ to graphically represent the tolerance regions for various individuals across confidence levels, utilizing the convex hull of extreme points for visualization purposes.

\subsubsection{Results}

Figure \ref{fig:firma} illustrates the handwriting trajectories alongside tolerance regions for confidence levels $\alpha=0.5$ and $\alpha=0.7$, including those from individuals with Parkinson's disease (yellow color). For individuals affected by Parkinson's, employing tolerance regions at $\alpha=0.5$ enables distinguishing over 75\% of the patients by examining whether their handwriting shape falls within or outside these predefined regions. These findings are encouraging, showcasing the potential of our uncertainty quantification approach in validating novel digital biomarkers for contemporary medical applications, such as analyzing handwriting patterns in the data from the clinical study examined.

\begin{figure}[ht!]
    \centering
    \includegraphics[width=0.8\textwidth]{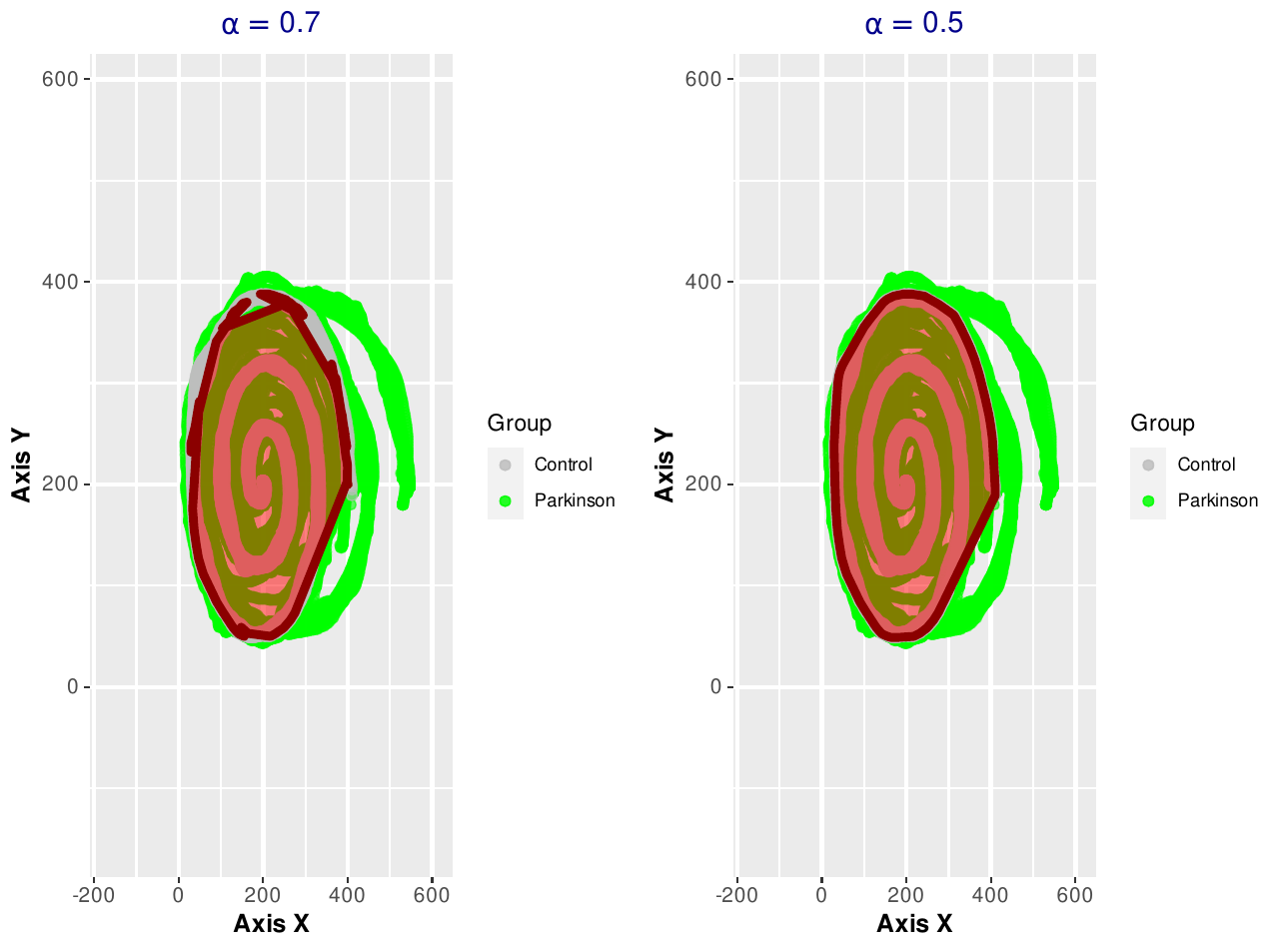}
\caption{Tolerance regions in expectations for various confidence levels in the written test using control individuals as a reference. The plots depict the raw trajectories of both control and Parkinson's patients. We do not introduce covariates in the derivation of prediction regions for the Fréchet mean.}
    \label{fig:firma}
\end{figure}

\subsection{Multidimensional Euclidean data}

\subsubsection{Motivation}

\noindent Multivariate Euclidean data, denoted as $\mathcal{Y} = \mathbb{R}^{m}$ $(m>1)$, is increasingly prevalent in clinical outcome analysis, particularly in the integration of multiple biomarkers to enhance diagnostic accuracy for various diseases. This approach is exemplified in diabetes research, where at least two biomarkers -- Fasting Plasma Glucose (FPG) and Glycosylated Hemoglobin (A1C) -- are considered in the diagnosis of the disease. In this simple $2-$dimensional example, providing naive uncertainty quantification measures is far from trivial due to the lack of a natural order in $\mathbb{R}^{2}$. Moreover, bivariate quantile estimation methods can depend on a specific notion of order tailored to each problem, as discussed in \cite{Klein2020}.

\noindent This section aims to illustrate our uncertainty quantification framework, which can accommodate any distance \(d_2\) to define the prediction region \(C^{\alpha}(\cdot)\), and overcome technical difficulties associated with defining a multivariate quantile.  To emphasize the mathematical intuition behind our methods for defining prediction regions using balls, we consider the standard Euclidean distance \(d_2^2(x, y) = \sum_{i=1}^{m}(x_i - y_i)^2\) for all \(x, y \in \mathbb{R}^m\). Continuing with the discussion on diabetes, we focus on a scientific application in this field.

\noindent To demonstrate the  application of our methodology, we consider the mentioned biomarkers: FPG and A1C as a response of the regression models that in this case will be a multivariate--response linear model. The goal is to create a predictive models to estimate the continuous outcomed used to diagnostic and monitoring the progression of diabetes mellitus disease.

\noindent Uncertainty analysis is crucial in clinical prediction models, as it helps understand the advantages and disadvantages of these models. In cases where there is significant uncertainty in the regression model $m$, it may be necessary to construct a more complex model incorporating new patient characteristic attributes. Alternatively, considering more sophisticated regression models, such as neural networks, might enhance the predictive capacity of the models.

	 
\subsubsection{Data Description}

\noindent We conducted an extensive study using a comprehensive dataset from the National Health and Nutrition Examination Survey
 (NHANES) cohort (see more details \cite{matabuena2022handgrip}), comprising a substantial number of individuals ($n = 58,972$).  For our analysis, we utilized predictor and response variables described in Table \ref{table:pacientes}, and applied a linear multivariate regression model to estimate the function $m$. We considered $k = 200$ as the number of neighbors in the kNN method.

\noindent Our study aimed to address two primary objectives: firstly, to demonstrate the remarkable computational efficiency of our algorithm, and secondly, to underscore the importance of quantifying uncertainty in standard disease risk scores. In an ideal clinical context, the primary clinical goal would be to establish a simple and cost-effective risk score that could accurately capture patients' glucose conditions without relying on traditional diabetes diagnostic variables, and stratify the risk of diabetes to promote new public health interventions \cite{matabuena2024deep}. With uncertainty quantification, we can evaluate whether our regression model $m$ can be useful for such clinical purposes.

\subsubsection{Results}

\noindent The implementation of our algorithm in the NHANES dataset in the \textbf{R} software exhibited remarkable efficiency, with execution times of just a few seconds.

\noindent Figure \ref{fig:diabetesnhanes} presents the results to apply the uncertainity quantification region in three patients. We
observe individual uncertainity differences when we predict the A1C and FPG biomarkers
 based on individual clinical characteristics. Patient $(C)$, characterized as young and healthy (see Table \ref{table:pacientes}), is predicted to be in the non-diabetes range with low uncertainty. In contrast, patient $(A)$, who is older, has a high waist circumference, and elevated systolic pressure, is predicted to fall within the diabetes and prediabetes range, despite exhibiting glucose profiles characteristic of a non-diabetic healthy individual. These outcomes underscore the need for more sophisticated variables to accurately predict the true glycemic patient status, as also evidenced by the marginal $R^{2}$ for each variable that takes values equal to $0.12$ and $0.08$, respectively for the A1C and FPG.

\noindent In separate research \cite{matabuena2024deep},
  we provide evidence that we need the incorporate at least one diabetes biomarker that directly measures the glucose of the individuals, such as FPG or A1C, is indispensable for building reliable predictive models. Essentially, this highlights that predicting Diabetes Mellitus is impossible without considering specific glucose measurements or proxy variables for them.

	\begin{figure}[ht!]
		\centering
		\includegraphics[width=0.9\linewidth]{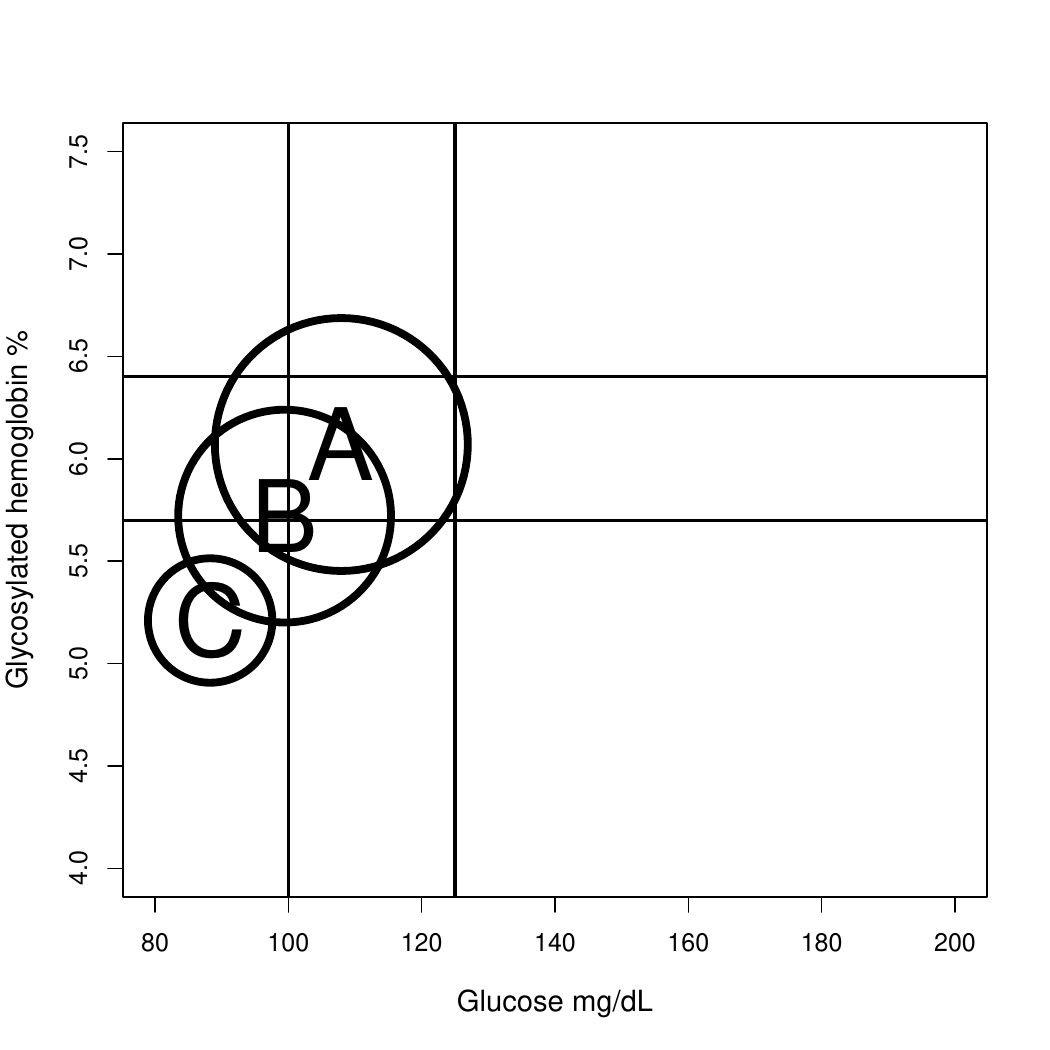}
\caption{Prediction regions for three patients based on the geometry of the Euclidean distance. The black lines in the plot indicate the transition between prediabetes and diabetes status. Patient (C) is normoglycemic, while the other two patients, (A) and (B), exhibit oscillations between two glycemic statuses. Patient (B) is characterized as normoglycemic and diabetic, while patient A shows diabetes and prediabetes. The plot highlights the varying glycemic patterns according to patient characteristics.}
		\label{fig:diabetesnhanes}
	\end{figure}

	\begin{table}[ht]
    \centering\
    \scalebox{0.8}{
    \small 
    \begin{tabular}{cccccccc}
        \toprule
        \textbf{ID} & \textbf{Age} & \textbf{Waist} & \textbf{Systolic Pressure} & \textbf{Diastolic Pressure} & \textbf{Triglycerides} & \textbf{Glucose} & \textbf{Glycosylated Hemoglobin} \\
        \midrule
        45424 (A) & 76 & 114.30 & 144 & 66 & 64 & 91.00 & 5.10 \\
        31858 (B) & 55 & 107.70 & 128 & 70 & 72 & 98.00 & 7.60 \\
        28159 (C) & 23 & 84.60 & 124 & 64 & 147 & 81.00 & 5.40 \\
        \bottomrule
    \end{tabular}}
    \caption{Characteristics of three patients in a real-life example where the response variable is multivariate (bivariate) Euclidean data.}
    \label{table:pacientes}
\end{table}
	
	\subsection{Distributional Representation of Biosensor Time Series}

\subsubsection{Motivation}

\begin{figure}[ht!]
        \includegraphics[width=\linewidth]{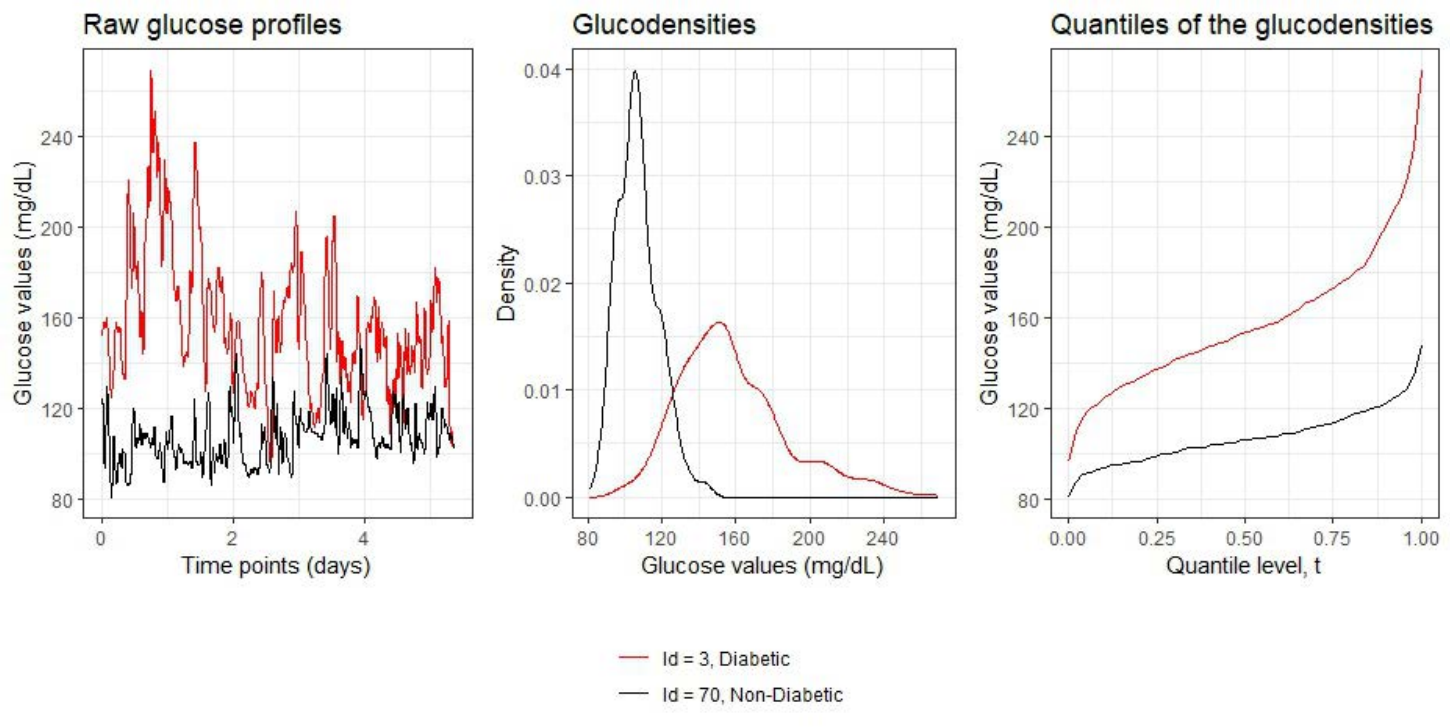}
    \caption{The figure presents the process of transformation CGM time series of two individuals in their marginal density functions and marginal quantile functions.}
    \label{fig:glucosetrans}
\end{figure}

\noindent The field of distributional data analysis has seen significant growth  within data--science communities \cite{ klein2023distributional, matabuena2022kernel,pmlr-v162-meunier22b}, largely due to its critical role in understanding  the marginal distributional characteristics of biological time series in regression modeling. We investigate the probability distributions within continuous glucose monitor (CGM) time series data \cite{doi:10.1177/0962280221998064}, vital for discerning distributional patterns like hypoglycemia,  and hyperglycemia in diabetes research. Figure \ref{fig:glucosetrans} illustrates such mathematical transformation in the case of a one non-diabetic and diabetic individual

\noindent We focus on analyzing  biological outcomes defined in the  space $\mathcal{Y}= \mathcal{W}_{2}(T)
$, comprising univariate continuous probability distributions on the compact support $T = [40, 400] \subset \mathbb{R}$. This work emphasizes "distributional representation" through quantile functions, aiming to enrich our understanding of glucose profiles' impacts from various factors beyond CGM simple summary as  CGM mean and -- standard deviation values estimated with the raw  glucose time series.

\noindent Given a random sample of healthy individuals, the goal of our analysis is to provide the reference quantile glucose values with our uncertainty quantification algorithm across different age groups. The determination of this reference profile can have various applications, such as detecting individuals at high risk of diabetes mellitus, as well as determining threshold criteria considering the latest advancements about glucose homeostasis provided by a CGM monitor. We  note that the glucose values at least, in average glucose levels increase with the age, and according to this biological behaviour, we must move to a medicine  based on personalized criteria. The use of CGM one distributional representations allows to examine this biological  phenomenon in all, low, middle and high-glucose concentrations.

\subsubsection{Data Description}

Data from \cite{shah2019continuous}, which collected glucose values at $5$-minute intervals over several days from non-diabetic individuals in everyday settings, was considered in our analysis. This data was originally collected to establish normative glucose ranges with simple summaries as CGM mean values by age. Extending this foundation, we adopt the distributional perspective of CGM data as our regression model's response, utilizing the $2$-Wasserstein distance (equivalently, quantile representation) in Global Fréchet regression. For analysis and visualization, we apply the supreme distance $d_2(\cdot,\cdot)= \norm{\cdot-\cdot}_{\infty}$, setting the confidence level at $\alpha=0.2$ and the kNN smoothing parameter to $k=30$.

\subsubsection{Results}

The application of our uncertainty quantification algorithm  shows a variation in uncertainty levels across different age groups, as outlined in Figure \ref{fig:glucoseref}. However, 
the analysis reveals a consistent conditional mean of glucose across ages but a marked increase in the uncertainty of quantile glucose values with the increase of the  age of patients. These clinical findings enhance our capability to classify and monitor health in non-diabetics through CGM. This approach to defining disease states offers a pathway towards personalized medicine, as evidenced by its application in our research on personalized sarcopenia definitions \cite{matabuena2022handgrip}, where the response variable was finite-dimensional outcome ($\mathcal{Y}= \mathbb{R}^{2}$).

\begin{figure}[H]
    \centering
    \begin{minipage}[b]{0.45\textwidth}
        \includegraphics[width=\linewidth]{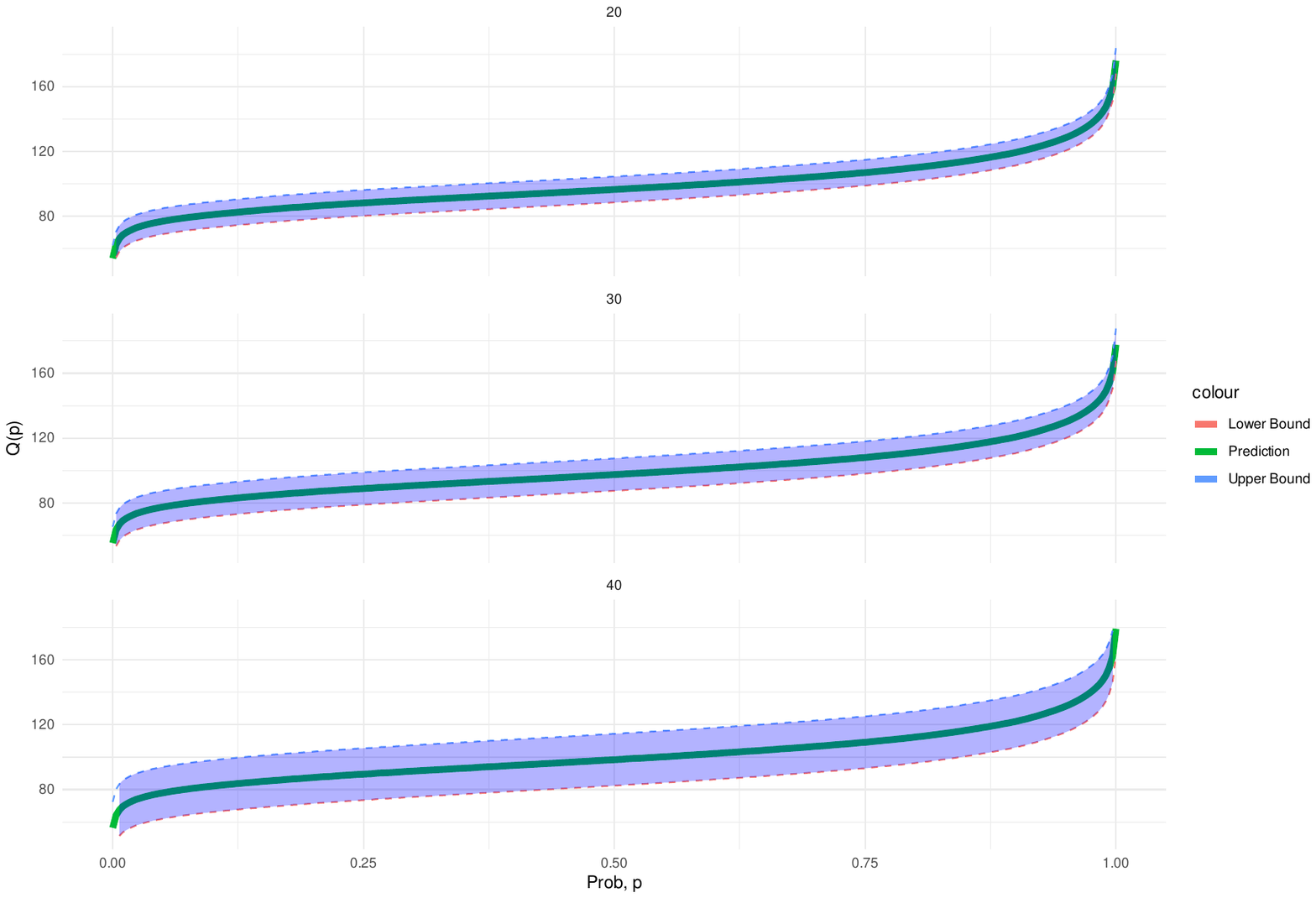}
        \caption*{(a) Middle-age individuals}
    \end{minipage}
    \hfill
    \begin{minipage}[b]{0.45\textwidth}
        \includegraphics[width=\linewidth]{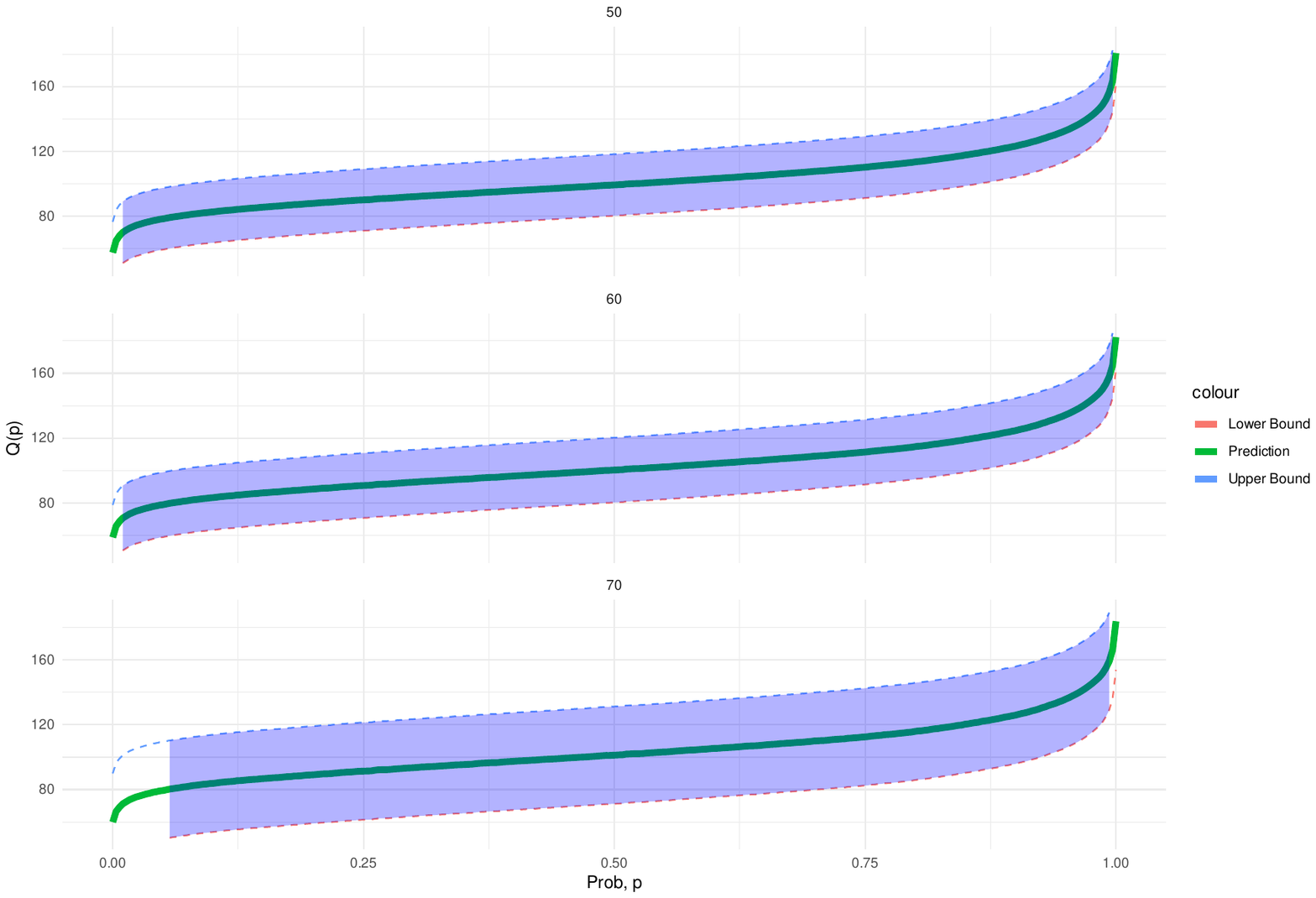}
        \caption*{(b) Elderly individuals}
    \end{minipage}
    \caption{The figure presents the normative continuous glucose monitoring (CGM) distributional values across different age groups, with a confidence level \(\alpha=0.2\). Our analysis indicates that while the conditional mean values remain similar across different age groups, there is an increase in uncertainty as age increases.}
    \label{fig:glucoseref}
\end{figure}



	\subsection{Neuroimage Applications: Laplacian Graphs}

\subsubsection{Motivation}

Neuroimaging represents  important field that is motivated by the development of  functional and metric space statistical  methods in order to improve or udnerstanding  the behaviour of the human brain
\cite{dubey2020functional,dai2021modeling, bhattacharjee2021concurrent, chen2021modeling}. Predominantly, functional magnetic resonance imaging (fMRI) serves as the benchmark for analyzing brain structures \cite{wang2021statistical,chen2021modeling,pervaiz2020optimising}.

\noindent Summarizing fMRI data's unique brain structure signatures involves generating individual profiles. These profiles estimate intra- and inter-regional cerebral connections through both continuous (correlation matrices) and discrete (graph structures) data representations \cite{wang2021statistical}.

\noindent
There is a growing interest in applying uncertainty quantification techniques within neuroimaging. Yet, these methods predominantly employ univariate approaches, highlighting a need for more comprehensive models \cite{ernsting2023groupdifferences}.

\noindent The goal of this analysis is to highlight the versatility of our algorithm to quantify the uncertainty over graphs. For each patient, a graph is avalaible that measures the connectivity between different brain regions.

\subsubsection{fMRI represention with Laplacian graphs}

We begin by examining the continuous domain of Pearson correlation matrices $\mathcal{Y}$ with dimension $r$, utilizing the Frobenius metric $d_{\text{FRO}}(A,B)= \sqrt{\sum_{i,j=1}^{r}(A_{ij}-B_{ij})^{2}}$. This approach extends to the network space, characterized by $r$ nodes and employing the same metric for Laplacian matrices. To align graph spaces with their Laplacian representations, certain technical prerequisites are necessary.

\noindent Consider a network $G_{m}= (V,E)$ with nodes $V= \{v_1,\dots, v_r\}$ and edge weights $E= \{w_{ij}: w_{ij}\geq 0; i,j=1,\dots,r\}$. We assume that $G_m$ is a simple graph, devoid of self-loops or multiple edges and  edge weights  bounded  $0\leq w_{ij}\leq w$.

\noindent \noindent We define the graph Laplacian as the matrix $L=\left(l_{ij}\right)$, where
\[
l_{ij}= 
\begin{cases}
-w_{ij}, & \text{if } i \neq j \\
\sum_{k \neq i} w_{ik}, & \text{if } i = j
\end{cases}
\]
for all indices $i, j = 1, \ldots, r$. This definition leads to the characterization of the space of networks through the associated space of graph Laplacians, defined as:
\[
\mathcal{L}_r = \left\{ L = \left(l_{ij}\right) : L = L^{\mathrm{T}}, L \mathbf{1}_r = \mathbf{0}_r, -W \leq l_{ij} \leq 0 \text{ for } i \neq j \right\},
\]

\noindent where $\mathbf{1}_{r}$ and $\mathbf{0}_{r}$ denote the $r$-vectors of ones and zeroes, respectively.





\subsubsection{Data Description}

Our study leverages an uncertainty quantification framework to analyze neuroimaging data, focusing on fMRI datasets from schizophrenia patients ($n=60$) and a control group ($n=45$) obtained from \cite{relion2019network}.

\noindent We utilized Laplacian graphs constructed from the data, following the preprocessing steps from \cite{relion2019network}. In our analysis, we consider a unique predictor $X\in \{0,1\}$, where $X=1$ indicates the schizophrenia of the individual.

\subsubsection{Results}

Utilizing Global Fréchet regression in the space of Laplacian networks $\mathcal{Y} = (\mathcal{L}_r, d_{\text{FRO}})$, which is equipped with the Frobenius distance applied to graph Laplacians, we estimated the regression function $m$. Subsequently, we applied the heteroscedastic uncertainty quantification algorithm with a fixed confidence level of $\alpha = 0.2$ and a smoothing parameter $k = 20$. Figure \ref{fig:neuro2} illustrates these conditional mean estimations along with their uncertainty intervals. The analysis reveals significant uncertainty, suggesting the potential utility of incorporating additional predictors to reduced the large uncertainity observed in the regression model.
	
	\begin{figure}[ht!]
		\centering
		\includegraphics[width=0.9\linewidth]{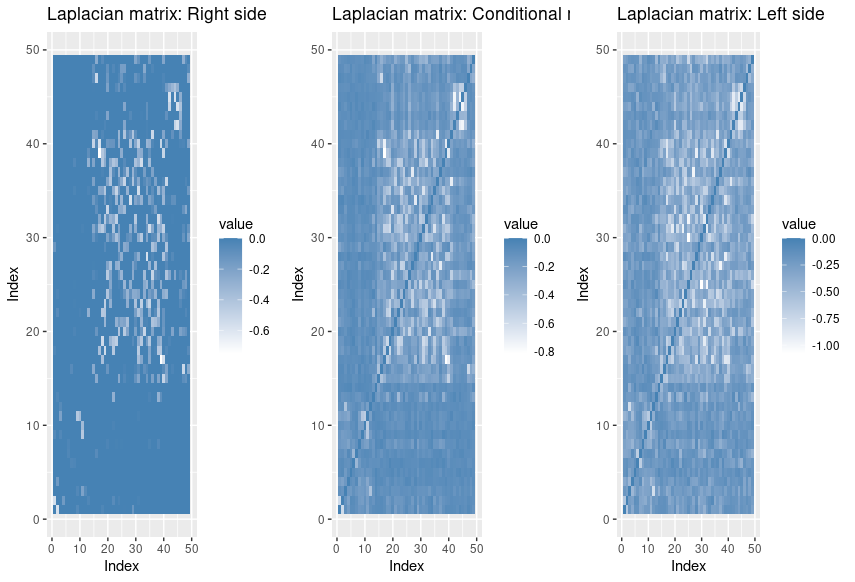}
		\caption{Left: Conditional mean function. Center: Left end of the uncertainty set. Right: Right end of the uncertainty set. Uncertainty analysis from a patient with schizophrenia using as a confidence level $20$ percent.}
		\label{fig:neuro2}
	\end{figure}

	\subsection{Variable selection in physical activity example}

\noindent	Physical exercise is widely recognized as an effective non-pharmacological intervention for a broad range of diseases, as well as a means of mitigating physiological decline with age and improving overall health \cite{harridge2017physical}. Therefore, understanding the patient attributes that contribute to variations in physical activity levels is of great scientific interest and can inform the development of public health policies aimed at reducing disparities in physical activity across different  subpopulations. In this study, we illustrate our variable selection methods for identifying relevant variables that are associated with physical activity levels of an individual.

	\subsubsection{Data Description}
	
	\noindent This study used data from the combined 2011-2014 NHANES cohort, which involved a comprehensive array of standardized health examinations and medical tests. These evaluations covered biochemical, dietary, physical activity, and anthropometric assessments. To measure physical activity levels at high resolution, the study utilized the ActiGraph GT3X+ accelerometer (Actigraph, Pensacola, FL) worn on the non-dominant wrist for seven full and two partial days among a subsample of participants. Further details regarding the specific protocols for collecting NHANES variables can be accessed on the official website at \url{https://www.cdc.gov/nchs/nhanes/index.htm}.

\noindent In our investigation, we adopt a functional perspective to summarize physical activity profiles using distributional quantile representations. These representations are equipped with the $2$ Wasserstein distance ($\mathcal{Y}= \mathcal{W}_{2}([a,b])$), as originally proposed in \cite{matabuena2021distributional}. Our predictor variables include gender, age, waist circumference, glucose levels, and a scalar summary of energetic expenditure, total activity count, and diet health score. We posited that the biochemical variable glucose would exhibit a weak association with physical activity levels, while the summary physical activity variables and diet would display a stronger relationship. Moreover, we anticipated that the remaining variables would demonstrate moderate connections with daily physical activity and energy expenditure.
	
	\subsubsection{Results}

\noindent The global importance of relevant and non-relevant variables for predicting distributional physical activity representations is summarized in Table \ref{table:nhanesacel}. Initially, based on our hypothesis, we expected glucose to be non-relevant, while waist circunference would exhibit a statistical association with the response. However, the results of the new variable selection method surprised us, indicating that waist circunference is not a significant variable. This unexpected outcome can be attributed to the fact that the effect of waist is already captured directly by the total activity count (TAC), as both variables are correlated.

\noindent To visually illustrate the local importance of the analyzed variables, we present Figure \ref{fig:sel}. The figure depicts the intensity of importance for each value of the predictor on the y-axis. Notably, in our application, the importance of the variables increases with higher values. This trend is particularly pronounced for the TAC variable, where larger TAC values correspond to higher variable importance.

\noindent Both Table \ref{table:nhanesacel} and Figure \ref{fig:sel} offer valuable insights into the relevance and importance of the variables in predicting distributional physical activity representations. These findings contribute to understanding the individual characteristics that impact physical activity profiles.

 \begin{table}[ht!]
    \centering
    \begin{tabular}{|c|c|c|c|}
        \hline
        \textbf{Variable No.} & \textbf{Variable Name} & \textbf{Selected} & \textbf{Raw p-value} \\
        \hline
        1 & \textbf{Gender} & TRUE & $0.003$ \\
        2 & \textbf{Age} & TRUE & $<0.001$ \\
        3 & \textbf{Waist} & FALSE & $<0.26$ \\
        4 & \textbf{TAC} & TRUE & $<0.001$ \\
        5 & \textbf{Healthy Eating Index} & TRUE & $<0.001$ \\
        6 & \textbf{Glucose Values} & FALSE & $<0.04$ \\
        \hline
    \end{tabular}
    \caption{Summary of results from applying a new variable selection method to predict distributional physical activity representations using data from the NHANES 2011-2014 dataset. Variable No. indicates the variable number in the Figure \ref{fig:sel}.}
    \label{table:nhanesacel}
\end{table}	
	\section{Discussion}
	
\noindent This paper introduces a novel data analysis framework designed to quantify uncertainty in regression models with random responses taking values in metric spaces. Our approach demonstrates  versatility, applicable to a wide range of regression algorithms, and operates under minimal theoretical conditions. Our method provides non-asymptotic guarantees in homoscedastic cases. Furthermore, this novel uncertainty quantification framework extends the notion of conditional quantile function for metric space responses \cite{liu2022quantiles}.

\noindent In the context of homoscedasticity, our approach represents a natural generalization of split conformal methods by incorporating the distance function $d_2$.  as a conformity score. For the heteroscedastic case, we draw upon recent work by Györfi and Walk on univariate responses \cite{gyofi2020nearest}, extending their methodology to metric spaces. Through extensive simulations and theoretical analysis, we validate the effectiveness of our method and highlight its properties in various settings.

\noindent We acknowledge that in the heteroscedastic case, the selection of model hyperparameters, specifically $d_2$, and $k$, smoothing parameter,
 can significantly impact the quality of final estimator results. To address this issue effectively, we propose exploring stability selection techniques \cite{https://doi.org/10.1111/j.1467-9868.2010.00740.x} and other resampling techniques based on bootstrap for the kNN method, as proposed in the literature \cite{hall2008choice, samworth2012optimal}. Investigating these approaches in future research is anticipated to offer valuable solutions to the tuning parameter problem.

\noindent The effectiveness of our approach is demonstrated through different relevant medical examples. For example, we present a novel application in diabetes, proposing normative distributional glucose values across different age groups using the notion of glucodensity \cite{doi:10.1177/0962280221998064}.  
Currently, the use of CGM devices in diabetes research is of increasing interest \cite{KESHET2023758}, and our applications have contributed to this direction in human health research.

\noindent Moreover,  we have developed a variable selection method for metric space responses. This method is computationally efficienct. Importantly, our approach integrates into other semiparametric models, such as the partial single index Fréchet model \cite{ghosal2023predicting}, distinguishing it from existing variable selection methods in metric spaces that are limited to the global Fréchet regression model \cite{tucker2021variable}.

\noindent Future research directions for our method are geared towards expanding its applicability and advancing its capabilities in handling diverse data scenarios and structures. One crucial aspect is extending the uncertainty quantification framework to effectively handle temporally and spatially correlated data, enabling us to tackle more complex real-world situations. Additional relevant direction could include  the accommodation of non-Euclidean predictors \cite{cohen2022metric}.

	\section*{Acknowledgments} 
	This research was  partially supported by US Horizon 2020 research and innovation
	program under grant agreement No. 952215 (TAILOR Connectivity Fund). Gábor Lugosi acknowledges the support of Ayudas Fundación BBVA a Proyectos de Investigación Cient\'ifica 2021 and the Spanish Ministry of Economy and Competitiveness, Grant PGC2018-101643-B-I00 and FEDER, EU.

	\bibliographystyle{apalike}
	\bibliography{manuscript}

	\appendix

	\newpage

	\section{Theory}

	\subsection{Homoscedastic case}
	
	\subsubsection{Statistical consistency}
	
	\begin{proposition}
    Assume that Assumption \ref{1Smet} is satisfied. Then
    \begin{equation}
        \frac{1}{n_2} \sum_{i \in [S_2]} \left| d_{2}(Y_i, m(X_i)) - d_{2}(Y_i, \widetilde{m}(X_i)) \right| = o_{p}(1).
    \end{equation}

 \noindent     Furthermore, for any  continuity point $v\in \mathbb{R}$ of $G$ where $G(v) = \mathbb{P}(d_{2}(Y, m(X)) \leq v)$, we have:
    \begin{equation}
        |\widetilde{G}^{*}(v) - G(v)| = o_{p}(1),
    \end{equation}
    where $\widetilde{G}^{*}(v) = \frac{1}{n_2} \sum_{i \in [S_2]} \mathbb{I}\{d_{2}(Y_i, \widetilde{m}(X_i)) \leq v\}$.
\end{proposition}

	\begin{proof}
\noindent		For the first result, observe that
		\begin{equation*}
		\frac{1}{n_2} \sum_{i\in [S_2]} |d_{2}(Y_i, m(X_i))-d_{2}(Y_i,\widetilde{m}(X_i))| \leq \frac{1}{n_2} \sum_{i\in [S_2]} |d_{2}(m(X_i),\widetilde{m}(X_i))|.
		\end{equation*}
\noindent		By Assumption \ref{1Smet} and the law of large numbers, we have
		\begin{equation*}
		\frac{1}{n_2} \sum_{i\in [S_2]} |d_{2}(Y_i, m(X_i))-d_{2}(Y_i, \widetilde{m}(X_i))| = o_{p}(1).
		\end{equation*}
		
\noindent		For the second part, fix $v \in \mathbb{R}$ as a continuity point. Define $\widetilde{G}_{m}^{*}(v)= \frac{1}{n_2} \sum_{i\in [S_2]} \mathbb{I}\{d_{2}(Y_i, m(X_i)) \leq v\}$ and the random quantity $R_{T}(v)= |\widetilde{G}^{*}_{m}(v)-G(v)|$, satisfying, by the law of large numbers, $R_{T}(v)= o_{p}(1)$.
		
\noindent		For any fixed \(\gamma > 0\), define the set \(A_{\gamma} = \{i \in [S_2] : |d_{2}(Y_i, \widetilde{m}(X_i)) - d_{2}(Y_i, m(X_i))| \geq \gamma \}\). Then,
		\begin{align*}
		& n_2\left[\widetilde{G}_{m}^{*}(v)-\widetilde{G}^{*}(v)\right] \\
		& \leq  \left| \sum_{i\in A_{\gamma}}  (\mathbb{I}\{d_{2}(Y_i, \widetilde{m}(X_i))\leq v\}-\mathbb{I}\{d_{2}(Y_i,m(X_i))\leq v\})\right| + \left| \sum_{i\in A_{\gamma}^{c}} (\mathbb{I}\{d_{2}(Y_i, \widetilde{m}(X_i)\leq v))\}-\mathbb{I}\{d_{2}(Y_i, m(X_i))\leq v\})\right| \\
		& \leq |A_{\gamma}| + \left|\sum_{i\in A_{\gamma}^{c}} (\mathbb{I}\{d_{2}(Y_i, \widetilde{m}(X_i))\leq v\}-\mathbb{I}\{d_{2}(Y_i, m(X_i))\leq v\})\right|.
		\end{align*}
\noindent  For $i\in A_{\gamma}^{c}$, $d_{2}(Y_i, m(X_i))-\gamma \leq d_{2}(Y_i, \widetilde{m}(X_i)) \leq d_{2}(Y_i, m(X_i))+\gamma$. Therefore,
		\begin{equation*}
		\sum_{i\in A_{\gamma}^{c}} \mathbb{I}\{d_{2}(Y_i, m(X_i))\leq v-\gamma\} \leq \sum_{i\in A_{\gamma}^{c}} \mathbb{I}\{d_{2}(Y_i, \widetilde{m}(X_i))\leq v\} \leq \sum_{i\in A_{\gamma}^{c}} \mathbb{I}\{d_{2}(Y_i, m(X_i))\leq v+\gamma\}.
		\end{equation*}
	\noindent	For any continuity point $v \in \mathbb{R}^{+}$ and any $\epsilon_{v} > 0$, there exists $\delta_{v} > 0$ such that for any $z \in (v - \delta_{v}, v + \delta_{v})$, $|G(z) - G(v)| < \epsilon_{v}$. Then, 
	
\begin{align*}
& n_2 \left| \widetilde{G}^{*}(v) - \widetilde{G}^{*}_{m}(v) \right| \\
&\leq  |A_{\gamma}|+ \left| \sum_{i\in A_{\gamma}^{c}} \mathbb{I}\{d_{2}(Y_i, \widetilde{m}(X_i)) \leq v \} - \sum_{i\in A_{\gamma}^{c}} \mathbb{I}\{d_{2}(Y_i, m(X_i)) \leq v \} \right| \\
&\leq  |A_{\gamma}|+  \left| n_2 \left[ \widetilde{G}_{m}^{*}(v + \gamma) - \widetilde{G}_{m}^{*}(v - \gamma)\right] - \left( \sum_{i\in A_{\gamma}} \mathbb{I}\{d_{2}(Y_i, m(X_i)) \leq v + \gamma \} - \sum_{i\in A_{\delta}} \mathbb{I}\{d_{2}(Y_i, m(X_i)) \leq v - \gamma\} \right) \right| \\
& \leq |A_{\gamma}|+ n_2 \left( (G(v + \gamma)) - G(v - \gamma) + R_{T}(v+\gamma)+ R_{T}(v-\gamma) \right) + |A_{\gamma}| \quad \text{(triangle inequality)} \\
& \leq  n_2 (2\epsilon_{v} + R_{T}(v+\gamma)+ R_{T}(v-\gamma)) + 2|A_{\gamma}|.
\end{align*}

\noindent First letting $n_{2}\to \infty$ and then $\gamma \to 0,$ we obtain

\begin{align*}
    |\widetilde{G}^{*}(v) - G(v)| \leq |\widetilde{G}^{*}(v) - \widetilde{G}_{m}(v)| + R_{T}(v) \leq o_{p}(1).
\end{align*}

	\end{proof}

\begin{corollary}\label{theo:quantil}
Assume that Assumption  \ref{1Smet} is satisfied.	For a fixed $\alpha \in (0,1)$, let $q_{1-\alpha} := \inf\{t \in \mathbb{R} : G(t) \geq 1-\alpha\}$, and suppose that $q_{1-\alpha}$ exists uniquely and is a continuity point of $G(\cdot)$. For any $\epsilon>0,$ as $n_1 \to \infty$, and,  $n_2 \to \infty$, we have
	\[
	\lim_{n_{2}\to \infty} \mathbb{P}\left(\left| \widetilde{q}^{m}_{1-\alpha} - q_{1-\alpha}\right|\geq \epsilon\right) = 1,
	\]
	and

	\[
	\lim_{n_{2}\to \infty} \mathbb{P}\left(\left| \widetilde{q}_{1-\alpha}  -  q_{1-\alpha}\right|\geq \epsilon\right) = 1,
	\]

 \noindent where  $\widetilde{q}^{m}_{1-\alpha}$ is  the empirical quantile from the empirical distribution
\[
\widetilde{G}_{m}(t) = \frac{1}{n_2} \sum_{i \in [S_2]} \mathbb{I}\{d_{2}(Y_i, m(X_i)) \leq t\},
\]

\noindent and $\widetilde{q}_{1-\alpha}$  is  the empirical quantile from the empirical distribution

\[
\widetilde{G}^{*}(t) = \frac{1}{n_2} \sum_{i \in [S_2]} \mathbb{I}\{d_{2}(Y_i, \widetilde{m}(X_i)) \leq t\}
.\]

\end{corollary}

\begin{proof}
	Both results are consequences of the  law of large numbers, since $\widetilde{G}^{*}(q_{1-\alpha}) = G(q_{1-\alpha}) + o_{p}(1)$ and $\widetilde{G}_{m}^{*}(q_{1-\alpha}) = G(q_{1-\alpha}) + o_{p}(1)$.
\end{proof}

	\subsection{Proof of Theorem \ref{th:conshom}}


 \noindent The goal of this section is to show:

\[
\int_{\mathcal{X}} \mathbb{P}(Y \in C^{\alpha}(x) \triangle \widetilde{C}^{\alpha}(x) \mid X=x, \mathcal{D}_{n}) P_{X}(dx) = o_{p}(1),
\]

\noindent where $C^{\alpha}(x)$ and $\widetilde{C}^{\alpha}(x)$ are defined in Section \ref{sec:homo}.

	\begin{proof}
		
\noindent For a fixed \(\alpha \in (0,1)\), let \(q_{1-\alpha}\) denote the corresponding quantile of $d_{2}(Y,m(X))$, and suppose that \(q_{1-\alpha}\) is a continuity point for the function \(G(\cdot)\). Define the sets \(\widetilde{C}_{m}^{\alpha}(x) = \mathcal{B}(m(x), \widetilde{q}_{1-\alpha})\) and \(\widetilde{C}_{q}^{\alpha}(x) = \mathcal{B}(\widetilde{m}(x), q_{1-\alpha})\), considering a fixed \(x \in \mathcal{X}\). Define also the function $G^{*}(t,x)= \mathbb{P}(d_{2}(Y, \widetilde{m}(X))\leq t \mid  X=x)$.
Given the properties of the metric induced by the symmetric difference of two sets, it holds that
		
	\begin{align*}
    & \mathbb{P}(Y \in C^{\alpha}(x) \triangle \widetilde{C}^{\alpha}(x) \mid X=x, \mathcal{D}_{n}) \\
    & \leq  \mathbb{P}(Y \in C^{\alpha}(x) \triangle \widetilde{C}_m^{\alpha}(x) \mid X=x, \mathcal{D}_{n}) + \mathbb{P}(Y \in \widetilde{C}_m^{\alpha}(x) \triangle \widetilde{C}^{\alpha}(x) \mid X=x, \mathcal{D}_{n}) \\
    & +  \mathbb{P}(Y \in \widetilde{C}^{\alpha}(x) \triangle \widetilde{C}_q^{\alpha}(x) \mid X=x, \mathcal{D}_{n}) + \mathbb{P}(Y \in \widetilde{C}_q^{\alpha}(x) \triangle \widetilde{C}^{\alpha}(x) \mid X=x, \mathcal{D}_{n}).
\end{align*}
		
	\noindent	We will analyze the four terms separately.

  \medskip
  \noindent
		\textbf{Case 1:}
		
\begin{eqnarray*}
\lefteqn{
    \mathbb{P}(Y \in C^{\alpha}(x) \triangle \widetilde{C}^{\alpha}_{m}(X) \mid X=x, \mathcal{D}_{n})   } \\
    & = & \mathbb{P}(\{Y: d_{2}(Y,m(X)) > q_{1-\alpha}, d_{2}(Y,m(X)) \leq \widetilde{q}_{1-\alpha} \} \mid X=x, \mathcal{D}_{n}) \\
    & & + \mathbb{P}(\{Y: d_{2}(Y,m(X)) \leq q_{1-\alpha}, d_{2}(Y,m(X)) > \widetilde{q}_{1-\alpha} \} \mid X=x, \mathcal{D}_{n}) \\
    & = & 2|G(\widetilde{q}_{1-\alpha})- G(q_{1-\alpha})|.
\end{eqnarray*}

\noindent Integrating, we have
\begin{align*}
    \int_{\mathcal{X}} & \mathbb{P}(Y \in C^{\alpha}(x) \triangle \widetilde{C}_m^{\alpha}(x) \mid X=x, \mathcal{D}_{n}) P_{X}(dx) \\
    & = 2|G(\widetilde{q}_{1-\alpha})- G(q_{1-\alpha})| \int_{\mathcal{X}} P_{X}(dx) = o_{p}(1),
\end{align*}
\noindent and in combination with Corollary \ref{theo:quantil}.

  \medskip
  \noindent
		\textbf{Case 2:}

\begin{align*}
& \hspace{0.1cm} \mathbb{P}(Y \in \widetilde{C}_m^{\alpha}(x) \triangle \widetilde{C}^{\alpha}(x) \mid X=x, \mathcal{D}_{n}) \\
& = \mathbb{P}(\{Y: d(Y,m(X)) > \widetilde{q}_{1-\alpha}, d_{2}(Y,\widetilde{m}(X)) \leq \widetilde{q}_{1-\alpha} \} \mid X=x, \mathcal{D}_{n}) \\
&  + \mathbb{P}(\{Y: d_{2}(Y,m(X)) \leq \widetilde{q}_{1-\alpha}, d_{2}(Y,\widetilde{m}(X)) > \widetilde{q}_{1-\alpha} \} \mid X=x, \mathcal{D}_{n}) \\
& \leq \mathbb{P}(\{Y: d_{2}(Y,\widetilde{m}(X)) \leq \widetilde{q}_{1-\alpha} < d_{2}(\widetilde{m}(X),m(X)) + d_{2}(y,\widetilde{m}(X)) \} \mid X=x, \mathcal{D}_{n}) \\
& + \mathbb{P}(\{Y: d_{2}(Y,m(X)) \leq \widetilde{q}_{1-\alpha} < d_{2}(\widetilde{m}(X),m(X)) + d_{2}(Y,m(X)) \} \mid X=x, \mathcal{D}_{n}) \\
& \leq G^{*}(\widetilde{q}_{1-\alpha} + d_{2}(\widetilde{m}(x),m(x)),x) - G^{*}(\widetilde{q}_{1-\alpha},x) + G(\widetilde{q}_{1-\alpha} + d_{2}(\widetilde{m}(x),m(x)) - G(\widetilde{q}_{1-\alpha}).
\end{align*}

\noindent Now, using the continuity hypothesis for $q_{1-\alpha}$, $\forall \epsilon>0,$ as $n_1 \to \infty$, and,  $n_2 \to \infty$, we have

\[
\lim_{n_{2}\to \infty} \mathbb{P}\left(\left| \widetilde{q}_{1-\alpha}  -  q_{1-\alpha}\right|\geq \epsilon\right) = 1,
\]

\noindent and that $\mathbb{E}(d_{2}(\widetilde{m}(X),m(X))\mid \mathcal{D}_{\text{train}})\to 0$ as $|\mathcal{D}_{\text{train}}|\to \infty$, we infer that $\int_{\mathcal{X}} G^{*}(\widetilde{q}_{1-\alpha} + d_{2}(\widetilde{m}(x),m(x)),x) - G^{*}(\widetilde{q}_{1-\alpha},x) P_{X}(dx) = o_{p}(1)$. Repeating for the other term, we have

\begin{align*}
& \int_{\mathcal{X}} \mathbb{P}(Y \in \widetilde{C}_m^{\alpha}(x) \triangle \widetilde{C}^{\alpha}(x) \mid X=x, \mathcal{D}_{n}) P_{X}(dx) \\
& =  \int_{\mathcal{X}} \left[ G^{*}(\widetilde{q}_{1-\alpha} + d_{2}(\widetilde{m}(x),m(x)),x) - G^{*}(\widetilde{q}_{1-\alpha},x) + G(\widetilde{q}_{1-\alpha} + d_{2}(\widetilde{m}(x),m(x)) - G(\widetilde{q}_{1-\alpha}) \right] P_{X} (dx)= o_{p}(1).
\end{align*}

\medskip
  \noindent		
		\textbf{Case 3:}
		
		\begin{align*}
		& \hspace{0.09cm} \mathbb{P}(Y \in \widetilde{C}^{\alpha}(x) \triangle \widetilde{C}_q^{\alpha}(x) \mid X=x, \mathcal{D}_{n}) \\
		& = \mathbb{P}(\{Y: d_{2}(Y,m(X)) > q_{1-\alpha}, d_{2}(Y,\widetilde{m}(X)) \leq q_{1-\alpha} \} \mid X=x, \mathcal{D}_{n}) 
		\\& + \hspace{0.085cm} \mathbb{P}(\{ Y: d_{2}(Y,m(X)) \leq q_{1-\alpha}, d(Y,\widetilde{m}(X)) > q_{1-\alpha}\} \mid X=x, \mathcal{D}_{n})
		\end{align*}
		
\noindent		This is similar to Case 2, except that we exchange the role of $\widetilde{q}_{1-\alpha}$ with $q_{1-\alpha}$. Therefore, we have
		
		\begin{align*}
		& \mathbb{P}(Y \in C^{\alpha}(x) \triangle \widetilde{C}_q^{\alpha}(x) \mid X=x, \mathcal{D}_{n}) \\
		& \leq G^{*}(q_{1-\alpha} + d_{2}(\widetilde{m}(x),m(x),x) - G^{*}(q_{1-\alpha},x) + G(q_{1-\alpha} + d_{2}(\widetilde{m}(x),m(x)),x) - G(q_{1-\alpha})
		\end{align*}
		
	\noindent	As a consequence,
		
		\begin{align*}
		& \int_{\mathcal{X}} \left[ G^{*}(q_{1-\alpha} + d_{2}(\widetilde{m}(x),m(x),x) - G^{*}(q_{1-\alpha},x) + G(q_{1-\alpha} + d_{2}(\widetilde{m}(x),m(x)) - G(q_{1-\alpha})) \right]  P_{X} (dx)=o_p(1).
		\end{align*}
		
	\medskip
  \noindent	
	\textbf{Case 4:}
	
\noindent	Following the arguments of Case $1$, we have that 
	
	\begin{align*}
	& \mathbb{P}(Y \in \widetilde{C}_q^{\alpha}(x) \triangle \widetilde{C}^{\alpha}(x) \mid X=x, \mathcal{D}_{n}) \\
	& = \mathbb{P}(\{ d_{2}(Y,\widetilde{m}(x)) > q_{1-\alpha}, d_{2}(Y, \widetilde{m}(x)) \leq \widetilde{q}_{1-\alpha} \} \mid X=x, \mathcal{D}_{n}) \\
	& + \hspace{0.075cm} \mathbb{P}(\{Y: d_{2}(Y,\widetilde{m}(X)) \leq q_{1-\alpha}, d_{2}(Y,\widetilde{m}(X)) > \widetilde{q}_{1-\alpha} \} \mid X=x, \mathcal{D}_{n}),
	\end{align*}
	
\noindent	and,
	
	\begin{align*}
	& \mathbb{P}(Y \in \widetilde{C}_q^{\alpha}(x) \triangle \widetilde{C}^{\alpha}(x) \mid X=x, \mathcal{D}_{n}) \leq 2 \mid G^{*}(\widetilde{q}_{1-\alpha},x)- G^{*}(q_{1-\alpha},x)| = o_p(1),
	\end{align*}
	
	\noindent and as a consequence
	
	\begin{align*}
	\int_{\mathcal{X}} & \mathbb{P}(Y \in \widetilde{C}_q^{\alpha}(x) \triangle \widetilde{C}^{\alpha}(x) \mid X=x, \mathcal{D}_{n}) P_{X} (dx) = o_{p}(1).
	\end{align*}
	
\noindent	Finally, combining the prior four results, we infer that
	
	\begin{align*}
	\int_{\mathcal{X}} & \mathbb{P}(Y \in C^{\alpha}(x) \triangle \widetilde{C}^{\alpha}(x)\mid X=x, \mathcal{D}_{n}) P_{X}(dx) = o_{p}(1).
	\end{align*}

	\end{proof}

		\subsection{Proof of Propositon \ref{the:minimax1}}

\begin{proof}
  \noindent  We begin by observing that, for the proof Proposition $4$, we can decompose the probability of interest as follows:
    
    \begin{align*}
        & \mathbb{P}(Y \in C^{\alpha}(x) \triangle \widetilde{C}^{\alpha}(x) \mid X=x, \mathcal{D}_{n})   \\ & \leq \mathbb{P}(Y \in C^{\alpha}(x) \triangle \widetilde{C}_m^{\alpha}(x) \mid X=x, \mathcal{D}_{n})+ \mathbb{P}(Y \in \widetilde{C}_m^{\alpha}(x) \triangle \widetilde{C}^{\alpha}(x) \mid  X=x, \mathcal{D}_{n}) \quad \\ & + \mathbb{P}(Y \in \widetilde{C}^{\alpha}(x) \triangle \widetilde{C}_q^{\alpha}(x) \mid X=x, \mathcal{D}_{n}) + \mathbb{P}(Y \in \widetilde{C}_q^{\alpha}(x) \triangle \widetilde{C}^{\alpha}(x) \mid X=x, \mathcal{D}_{n}).
    \end{align*}

\noindent    Now, for a fixed $x\in \mathcal{X},$ we decompose the four terms explicitly:

    \begin{enumerate}
        \item $\mathbb{P}(Y \in C^{\alpha}(x) \triangle \widetilde{C}_m^{\alpha}(x) \mid X=x, \mathcal{D}_{n})= 2|G(\widetilde{q}^{m}_{1-\alpha})- G(q_{1-\alpha})|$
        \item $\mathbb{P}(Y \in \widetilde{C}_m^{\alpha}(x) \triangle \widetilde{C}^{\alpha}(x)\mid  X=x, \mathcal{D}_{n}) \\
        \leq G^{*}(\widetilde{q}_{1-\alpha} + d_{2}(\widetilde{m}(x),m(x)),x) - G^{*}(\widetilde{q}_{1-\alpha},x) + G(\widetilde{q}_{1-\alpha} + d_{2}(\widetilde{m}(x),m(x))) - G(\widetilde{q}_{1-\alpha})$
        \item $\mathbb{P}(Y \in C^{\alpha}(x) \triangle \widetilde{C}_q^{\alpha}(x) | X=x, \mathcal{D}_{n}) \\
        \leq G^{*}(q_{1-\alpha} + d_{2}(\widetilde{m}(x),m(x)),x) - G^{*}(q_{1-\alpha},x) + G(q_{1-\alpha} + d_{2}(\widetilde{m}(x),m(x))) - G(q_{1-\alpha})$
        \item $\mathbb{P}(Y \in \widetilde{C}_q^{\alpha}(x) \triangle \widetilde{C}^{\alpha}(x) \mid X=x, \mathcal{D}_{n}) = 2 |G^{*}(\widetilde{q}_{1-\alpha},x) - G^{*}(q_{1-\alpha},x)|$
    \end{enumerate}

\noindent    First of all, we note that for the second and third terms, we use the fact that $G$ and $G^{*}$  are Lipschitz functions, and that we obtain an upper bound of the form $C d_{2}(\widetilde{m}(x),m(x))$.

\noindent    For the first and fourth terms, we can provide an upper bound in terms of the distance between the population quantile and the empirical quantile. More specifically, we have 
    
    \[
    \mathbb{P}(Y \in C^{\alpha}(x) \triangle \widetilde{C}_m^{\alpha}(x) \mid X=x, \mathcal{D}_{n}) + \mathbb{P}(Y \in \widetilde{C}_q^{\alpha}(x) \triangle \widetilde{C}^{\alpha}(x) \mid X=x, \mathcal{D}_{n}) \leq 4C |\widetilde{q}_{1-\alpha} - q_{1-\alpha}|.
    \]

  \noindent  Now,  integrating $\int_{\mathcal{X}} \mathbb{P}(Y \in C^{\alpha}(x) \triangle \widetilde{C}^{\alpha}(x) \mid X=x, \mathcal{D}_{n})P_{X}(dx)$, we obtain the desired results.
\end{proof}
					
	\subsection{Heteroscedastic case}
	
	\noindent To extend consisting results to the heteroscedastic scenario, we employ analogous arguments as in the homoscedastic case. The primary challenge arises from our use of a kNN algorithm to locally estimate the conditional distribution of distances, $G(v,x)= \mathbb{P}(Y\in d_{2}(Y,m(X))\leq v \mid  X=x)$, as applied in \cite{gyofi2020nearest}. For a given $x \in \mathbb{R}^{p}$ and $v \in \mathbb{R}$ as a continuity point, establishing convergence in probability crucially relies on the expression:

\begin{equation}
    R_{T}(v,x) = |\widetilde{G}^{*}_{m}(v,x) - G(v,x)| = o_{p}(1), 
\end{equation}

\noindent where $\widetilde{G}_{m}^{*}(v,x)= \frac{1}{k} \sum_{i\in N_{k}(x)} \mathbb{I}\{d_{2}(Y_i, m(X_i)) \leq v\}$.

\noindent To address these challenges, we first establish some preliminary results.

\begin{theorem}[Strong consistency of kNN \cite{10.1214/aos/1176345647}]\label{thm:aux}
 \noindent   Let $(X,Y)$, $(X_{1},Y_{1}),\dots,(X_{n},Y_{n})$ be independent identically distributed $\mathbb{R}^{p}\times \mathbb{R}$-valued random vectors with $\mathbb{E}(|Y|)<\infty$. The regression function $m(x)= \mathbb{E}(Y|X=x)$, for $x\in \mathbb{R}^{p}$, is estimated by
    \begin{equation}
        \widetilde{m}(x)=   \frac{1}{k} \sum_{i\in N_{k}(x)} Y_{i}, \hspace{0.2cm} k\in \mathbb{N}.
    \end{equation}

   \noindent Assume also that
    \begin{equation*}
        |Y|\leq \gamma <\infty
    \end{equation*}
 \noindent   and $k= k_{n}$ is a sequence of integers such that $\frac{k}{n}\to 0$ and $k\to \infty$ as $n\to \infty$.

 \noindent   Then, for the nearest neighbor estimate, $\mathbb{E}\{|\widetilde{m}_{n}(x)-m(x)|\}\to 0$ as $n\to \infty$ for almost all $x\in \mathcal{X}$, and $\mathbb{E}\{|\widetilde{m}_{n}(X)-m(X)|\}\to 0$ as $n\to \infty$. If, in addition, $\frac{k}{\log n}\to \infty$ as $k, n\to \infty$, then
    
    \begin{equation*}
     \noindent   \mathbb{E}\{|\widetilde{m}_{n}(X)-m(X)||X_{1},Y_{1},\dots,X_{n},Y_{n}\}\to 0 \text{ a.s.}
    \end{equation*}
\end{theorem}

\begin{lemma}\label{lemma:new2o}
    \noindent Suppose that $k\to \infty$ and $\frac{k}{log n}\to \infty$ as $n\to \infty$ and hypothesis of Theorem \ref{theorem:hetero} are satisfied. Then for almost every $x\in \mathbb{R}^{p},$
    \begin{equation*}
        |\widetilde{G}^{*}_{m}(v,x) - G(v,x)| = o_{p}(1)  \text{ as }  n \to \infty,
    \end{equation*}
    \noindent and
    \begin{equation*}
        \mathbb{E}\{|\widetilde{G}^{*}_{m}(v,x) - G(v,x)|X_{1},Y_{1},\dots,X_{n},Y_{n}\} = o_{p}(1)  \text{ as } n \to \infty,
    \end{equation*}
   \noindent where $\widetilde{G}_{m}^{*}(v,x) = \frac{1}{k} \sum_{i \in N_{k}(x)} \mathbb{I}\{d_{2}(Y_i, m(X_i)) \leq v\}$.
\end{lemma}

\begin{proof}
    \noindent For a fixed positive $v$ continuity point of $G(v,x)=\mathbb{P}(d_{2}(Y,m(X))\leq v \mid X=x)$, define the function $f:\mathbb{R}^{p}\times \mathbb{R} \to [0,1]$ as $f(x,y) = \mathbb{I}\{d_{2}(y,m(x))\leq v\}$. Consider
    \begin{equation*}
        \widetilde{m}(x) = \frac{1}{k} \sum_{i=1}^{k} f(X_{(i)}(x),Y_{(i)}(x)) = \widetilde{G}_{m}^{*}(v,x).
    \end{equation*}
    \noindent Note that
    \begin{equation}
        |\widetilde{G}^{*}_{m}(v,x) - G(v,x)| = \left|\frac{1}{k} \sum_{i=1}^{k} f(X_{(i)}(x),Y_{(i)}(x)) - \mathbb{E}[f(X,Y)\mid X=x]\right|.
    \end{equation}
    \noindent Now, define the random variable $Y^{\prime} = f(X,Y)$, and the pair $(X, Y^{\prime}) \in \mathbb{R}^{p}\times  \mathbb{R}$ satisfies the conditions of Theorem \ref{thm:aux}, where we obtain:
    \begin{equation}
        |\widetilde{G}^{*}_{m}(v,x) - G(v,x)| =  o_{p}(1) \text{ as }   n \to \infty,  
    \end{equation}
\noindent    and
    \begin{equation*}
        \mathbb{E}\{|\widetilde{G}^{*}_{m}(v,x) - G(v,x)\mid X_{1},Y_{1},\dots,X_{n},Y_{n}\} \to 0 \text{ a.s as } n \to \infty.
    \end{equation*}
\end{proof}

	\subsubsection{Statistical consistency}
	\begin{proposition}
  \noindent   Assume that Assumption \ref{assu} holds. For every fixed $x\in \mathcal{X},$ and $v\in \mathbb{R}$ is a continuity point of $G(v,x) = \mathbb{P}(d_{2}(Y,m(X))\leq v\mid X=x)$, and $k\in \mathbb{N}
$ satisfying the regularity conditions from Theorem \ref{theorem:hetero}, then    
    \begin{equation}
        |\widetilde{G}^{*}(v,x) - G(v,x)| = o_{p}(1),
    \end{equation}
    \noindent where $\tilde{G}^{*}(v,x) = \frac{1}{k} \sum_{i \in N_{k}(x)} \mathbb{I}\{d_{2}(Y_i, \widetilde{m}(X_i)) \leq v\}$.
\end{proposition}

\begin{proof}

\noindent Fix $v \in \mathbb{R}$ as a continuity point. Define $\widetilde{G}_{m}^{*}(v,x)= \frac{1}{k} \sum_{i \in N_{k}(x)} \mathbb{I}\{d_{2}(Y_i, m(X_i)) \leq v\}$ and $R_{T}(v,x)= |\widetilde{G}^{*}_{m}(v,x)-G(v,x)|$. By Lemma \ref{lemma:new2o} $R_{T}(v,x)=o_{p}(1),$ for almost every $x\in \mathbb{R}^{p}.$

\noindent		For any fixed \(\gamma > 0\), define the set \(A_{\gamma} = \{i \in N_{k}(x) : |d_{2}(Y_i, \widetilde{m}(X_i)) - d_{2}(Y_i, m(X_i))| \geq \gamma \}\). Then,
		\begin{align*}
		& k\left[\widetilde{G}_{m}^{*}(v,x)-\widetilde{G}^{*}(v,x)\right] \\
		& \leq  \left| \sum_{i\in \in A_{\gamma}}  (\mathbb{I}\{d_{2}(Y_i, \widetilde{m}(X_i))\leq v\}-\mathbb{I}\{d_{2}(Y_i,m(X_i))\leq v\})\right| + \left| \sum_{i\in A_{\gamma}^{c}} (\mathbb{I}\{d_{2}(Y_i, \widetilde{m}(X_i)\leq v))\}-\mathbb{I}\{d_{2}(Y_i, m(X_i))\leq v\})\right| \\
		& \leq |A_{\gamma}| + \left|\sum_{i\in A_{\gamma}^{c}} (\mathbb{I}\{d_{2}(Y_i, \widetilde{m}(X_i))\leq v\}-\mathbb{I}\{d_{2}(Y_i, m(X_i))\leq v\})\right|.
		\end{align*}
\noindent  For $i\in A_{\gamma}^{c}$, $d_{2}(Y_i, m(X_i))-\gamma \leq d_{2}(Y_i, \widetilde{m}(X_i)) \leq d_{2}(Y_i, m(X_i))+\gamma$. Therefore,
		\begin{equation*}
		\sum_{i\in A_{\gamma}^{c}} \mathbb{I}\{d_{2}(Y_i, m(X_i))\leq v-\gamma\} \leq \sum_{i\in A_{\gamma}^{c}} \mathbb{I}\{d_{2}(Y_i, \widetilde{m}(X_i))\leq v\} \leq \sum_{i\in A_{\gamma}^{c}} \mathbb{I}\{d_{2}(Y_i, m(X_i))\leq v+\gamma\}.
		\end{equation*}
	\noindent For a fixed $x\in \mathbb{R}^{p}$, and
 	 any continuity point $v \in \mathbb{R}^{+}$ and any $\epsilon_{v} > 0$, there exists $\delta_{v} > 0$ such that for any $z \in (v - \delta_{v}, v + \delta_{v})$, $|G(z,x) - G(v,x)| < \epsilon_{v}$. Then, 
	
\begin{align*}
& k \left| \widetilde{G}^{*}(v,x) - \widetilde{G}_{m}(v,x) \right| \\
&\leq  |A_{\gamma}|+ \left| \sum_{i\in A_{\gamma}^{c}} \mathbb{I}\{d_{2}(Y_i, \widetilde{m}(X_i)) \leq v \} - \sum_{i\in A_{\gamma}^{c}} \mathbb{I}\{d_{2}(Y_i, m(X_i)) \leq v \} \right| \\
&\leq  |A_{\gamma}|+  \left| k \left[ \widetilde{G}_{m}^{*}(v + \gamma,x) - \widetilde{G}_{m}(v - \gamma,x)\right] - \left( \sum_{i\in A_{\gamma}} \mathbb{I}\{d_{2}(Y_i, m(X_i)) \leq v + \gamma \} - \sum_{i\in A_{\delta}} \mathbb{I}\{d_{2}(Y_i, m(X_i)) \leq v - \gamma\} \right) \right| \\
& \leq |A_{\gamma}|+ k \left( (G(v + \gamma,x)) - G(v - \gamma,x) + R_{T}(v+\gamma,x)+ R_{T}(v-\gamma,x) \right) + |A_{\gamma}| \quad \text{(triangle inequality)} \\
& \leq  k (2\epsilon_{v} + R_{T}(v+\gamma,x)+ R_{T}(v-\gamma,x)) + 2|A_{\gamma}|.
\end{align*}

\noindent First letting $n_{2}\to \infty,$
$k\to \infty$ and $\gamma \to 0,$ we obtain

\begin{align*}
    |\widetilde{G}^{*}(v,x) - G(v,x)| \leq |\widetilde{G}^{*}(v,x) - \widetilde{G}_{m}(v,x)| + R_{T}(v,x) \leq o_{p}(1).
\end{align*}

\end{proof}

 	\begin{corollary}\label{theo:quantilknn}
Assume that Assumption  \ref{assu} is satisfied.	For a fixed $x\in \mathcal{X}$, and
$\alpha \in (0,1)$, let $q_{1-\alpha}(x) := \inf\{t \in \mathbb{R} : G(t,x) \geq 1-\alpha\}$, and suppose that $q_{1-\alpha}(x)$ exists uniquely and is a continuity point of $G(\cdot,x)$. $\forall \epsilon>0,$ as $n_1 \to \infty$, $n_2 \to \infty$, $k\to \infty
$ we have
	\[
\lim_{k\to \infty}	\lim_{n_{2}\to \infty} \mathbb{P}\left(\left| \widetilde{q}^{m}_{1-\alpha}(x) - q_{1-\alpha}(x)\right|\geq \epsilon\right) = 1,
	\]
	and

	\[
	\lim_{k\to \infty} \lim_{n_{2}\to \infty} \mathbb{P}\left(\left| \widetilde{q}_{1-\alpha}(x)  -  q_{1-\alpha}(x)\right|\geq \epsilon\right) = 1.
	\]

\end{corollary}	
	
	\subsection{Proof of Theorem \ref{theorem:hetero}}


 \noindent The goal of this proof is to show:

\[
\int_{\mathcal{X}} \mathbb{P}(Y \in C^{\alpha}(x) \triangle \widetilde{C}^{\alpha}(x) \mid X=x, \mathcal{D}_{n}) P_{X}(dx) = o_{p}(1),
\]

\noindent where $C^{\alpha}(x)$ and $\widetilde{C}^{\alpha}(x)$ are defined as before.

	\begin{proof}
		
\noindent For a fixed $x\in \mathbb{R}^{p},$ and \(\alpha \in (0,1)\), let \(q_{1-\alpha}(x)\) denote the corresponding quantile of $d_{2}(Y,m(X))\mid X=x$, and suppose that \(q_{1-\alpha}(x)\) is a continuity point for the function \(G(\cdot,x)\). Define the sets \(\widetilde{C}_{m}^{\alpha}(x) = \mathcal{B}(m(x), \widetilde{q}_{1-\alpha}(x))\) and \(\widetilde{C}_{q}^{\alpha}(x) = \mathcal{B}(\widetilde{m}(x), q_{1-\alpha}(x) )\), considering a fixed \(x \in \mathcal{X}\). Define also the function $G^{*}(t,x)= \mathbb{P}(d_{2}(Y, \widetilde{m}(X))\leq t \mid  X=x)$.
Given the properties of the metric induced by the symmetric difference of two sets, it holds that 
		
	\begin{align*}
		& \hspace{0.1cm} \mathbb{P}(Y \in C^{\alpha}(x) \triangle \widetilde{C}^{\alpha}(x) \mid X=x, \mathcal{D}_{n}) 	\\& \leq  \mathbb{P}(Y \in C^{\alpha}(x) \triangle \widetilde{C}_m^{\alpha}(x) \mid X=x, \mathcal{D}_{n}) + \mathbb{P}(Y \in \widetilde{C}_m^{\alpha}(x) \triangle \widetilde{C}^{\alpha}(x) | X=x, \mathcal{D}_{n}) \\& + \mathbb{P}(Y \in C^{\alpha}(x) \triangle \widetilde{C}_q^{\alpha}(x) \mid X=x, \mathcal{D}_{n}) + \mathbb{P}(Y \in \widetilde{C}_q^{\alpha}(x) \triangle \widetilde{C}^{\alpha}(x) \mid X=x, \mathcal{D}_{n}).
		\end{align*}
		
	\noindent	We will analyze the four terms separately.

  \medskip
  \noindent
		\textbf{Case 1:}
		
\begin{eqnarray*}
\lefteqn{
    \mathbb{P}(Y \in C^{\alpha}(x) \triangle \widetilde{C}^{\alpha}_{m}(X) \mid X=x, \mathcal{D}_{n})   } \\
    & = & \mathbb{P}(\{Y: d_{2}(Y,m(X)) > q_{1-\alpha}(x), d_{2}(Y,m(X)) \leq \widetilde{q}_{1-\alpha}(x) \} \mid X=x, \mathcal{D}_{n}) \\
    & & + \mathbb{P}(\{Y: d_{2}(Y,m(X)) \leq q_{1-\alpha}(x), d_{2}(Y,m(X)) > \widetilde{q}_{1-\alpha}(x) \} \mid X=x, \mathcal{D}_{n}) \\
    & = & 2|G(\widetilde{q}_{1-\alpha}(x),x)- G(q_{1-\alpha}(x),x)|.
\end{eqnarray*}

\noindent Integrating, we have
\begin{align*}
    \int_{\mathcal{X}} & \mathbb{P}(Y \in C^{\alpha}(x) \triangle \widetilde{C}_m^{\alpha}(x) \mid X=x, \mathcal{D}_{n}) P_{X}(dx) \\
    & = 2|G(\widetilde{q}_{1-\alpha}(x),x)- G(q_{1-\alpha}(x),x)| \int_{\mathcal{X}} P_{X}(dx) = o_{p}(1),
\end{align*}
\noindent and in combination with Corollary \ref{theo:quantilknn}.

  \medskip
  \noindent
		\textbf{Case 2:}

\begin{align*}
& \mathbb{P}(Y \in \widetilde{C}_m^{\alpha}(x) \triangle \widetilde{C}^{\alpha}(x) \mid X=x, \mathcal{D}_{n}) \\
& = \mathbb{P}(\{Y: d(Y,m(X)) > \widetilde{q}_{1-\alpha}(x), d_{2}(Y,\widetilde{m}(X)) \leq \widetilde{q}_{1-\alpha}(x) \} \mid X=x, \mathcal{D}_{n})+ \\
&  \mathbb{P}(\{Y: d_{2}(Y,m(X)) \leq \widetilde{q}_{1-\alpha}(x), d_{2}(Y,\widetilde{m}(X)) > \widetilde{q}_{1-\alpha}(x) \} \mid X=x, \mathcal{D}_{n}) \\
& \leq \mathbb{P}(\{Y: d_{2}(Y,\widetilde{m}(X)) \leq \widetilde{q}_{1-\alpha}(x) < d_{2}(\widetilde{m}(X),m(X)) + d_{2}(y,\widetilde{m}(X)) \} \mid X=x, \mathcal{D}_{n}) +\\
& \mathbb{P}(\{Y: d_{2}(Y,m(X)) \leq \widetilde{q}_{1-\alpha}(x) < d_{2}(\widetilde{m}(X),m(X)) + d_{2}(Y,m(X)) \} \mid X=x, \mathcal{D}_{n}) \\
& \leq G^{*}(\widetilde{q}_{1-\alpha}(x) + d_{2}(\widetilde{m}(x),m(x)),x) - G^{*}(\widetilde{q}_{1-\alpha}(x),x) + G(\widetilde{q}_{1-\alpha}(x) + d_{2}(\widetilde{m}(x),m(x)),x) - G(\widetilde{q}_{1-\alpha}(x),x).
\end{align*}
\noindent Now, using the continuity hypothesis for $q_{1-\alpha}(x)$, $\forall \epsilon>0,$ as $n_1 \to \infty$, $n_2 \to \infty$, and $k\to\infty 
$ we have
\[
\lim_{k\to \infty} \lim_{n_{2}\to \infty} \mathbb{P}\left(\left| \widetilde{q}^{m}_{1-\alpha}(x) - q_{1-\alpha}(x)\right|\geq \epsilon\right) = 1,
\]
\noindent and
\[
\lim_{n_{2}\to \infty} \mathbb{P}\left(\left| \widetilde{q}_{1-\alpha}(x)  -  q_{1-\alpha}(x)\right|\geq \epsilon\right) = 1,
\]

\noindent and that $\mathbb{E}(d_{2}(\widetilde{m}(X),m(X))\mid \mathcal{D}_{\text{train}})\to 0$ as $|\mathcal{D}_{\text{train}}|\to \infty$, we infer that $\int_{\mathcal{X}} G^{*}(\widetilde{q}_{1-\alpha}(x) + d_{2}(\widetilde{m}(x),m(x)),x) - G^{*}(\widetilde{q}_{1-\alpha}(x),x) P_{X}(dx) = o_{p}(1)$. Repeat for the other term, we have

\begin{align*}
& \int_{\mathcal{X}} \mathbb{P}(Y \in \widetilde{C}_m^{\alpha}(x) \triangle \widetilde{C}^{\alpha}(x) \mid X=x, \mathcal{D}_{n}) P_{X}(dx) \\
& =  \int_{\mathcal{X}} [ G^{*}(\widetilde{q}_{1-\alpha}(x) + d_{2}(\widetilde{m}(x),m(x)),x) - G^{*}(\widetilde{q}_{1-\alpha}(x),x) \\ & + G(\widetilde{q}_{1-\alpha}(x) + d_{2}(\widetilde{m}(x),m(x),x) - G(\widetilde{q}_{1-\alpha}(x),x)] P_{X} (dx)= o_{p}(1).
\end{align*}

\medskip
  \noindent		
		\textbf{Case 3:}
		
		\begin{align*}
		& \mathbb{P}(Y \in C^{\alpha}(x) \triangle \widetilde{C}_q^{\alpha}(x) \mid X=x, \mathcal{D}_{n}) \\
		& = \mathbb{P}(\{Y: d_{2}(Y,m(X)) > q_{1-\alpha}(x), d_{2}(Y,\widetilde{m}(X)) \leq q_{1-\alpha}(x) \} \mid X=x, \mathcal{D}_{n}) \\
		& + \mathbb{P}(\{ Y: d_{2}(Y,m(X)) \leq q_{1-\alpha}(x), d(Y,\widetilde{m}(X)) > q_{1-\alpha}(x)\} \mid X=x, \mathcal{D}_{n}).
		\end{align*}
		
\noindent		This is similar to Case 2, except that we exchange the role of $\widetilde{q}_{1-\alpha}(x)$ with $q_{1-\alpha}(x)$. Therefore, we have
		
		\begin{align*}
		& \mathbb{P}(Y \in C^{\alpha}(x) \triangle \widetilde{C}_q^{\alpha}(x) \mid X=x, \mathcal{D}_{n}) \\
		& \leq G^{*}(q_{1-\alpha}(x) + d_{2}(\widetilde{m}(x),m(x),x) - G^{*}(q_{1-\alpha}(x),x) + G(q_{1-\alpha}(x) + d_{2}(\widetilde{m}(x),m(x),x) - G(q_{1-\alpha}(x),x).
		\end{align*}
		
	\noindent	As a consequence,
		
		\begin{align*}
		& \int_{\mathcal{X}} \left[ G^{*}(q_{1-\alpha}(x) + d_{2}(\widetilde{m}(x),m(x),x) - G^{*}(q_{1-\alpha}(x) + G(q_{1-\alpha}(x)+ d_{2}(\widetilde{m}(x),m(x),x) - G(q_{1-\alpha}(x),x) \right]  P_{X} (dx)=o_p(1).
		\end{align*}
		
	\medskip
  \noindent	
	\textbf{Case 4:}
	
\noindent	Following the arguments of Case $1$, we have that 
	
	\begin{align*}
	& \hspace{0.08cm} \mathbb{P}(Y \in \widetilde{C}_q^{\alpha}(x) \triangle \widetilde{C}^{\alpha}(x) \mid X=x, \mathcal{D}_{n}) \\
	& = \mathbb{P}(\{ d_{2}(Y,\widetilde{m}(X)) > q_{1-\alpha}(x), d_{2}(Y, \widetilde{m}(X)) \leq \widetilde{q}_{1-\alpha}(x) \} \mid X=x, \mathcal{D}_{n}) \\
	& + \mathbb{P}(\{Y: d_{2}(Y,\widetilde{m}(X)) \leq q_{1-\alpha}(x), d_{2}(Y,\widetilde{m}(X)) > \widetilde{q}_{1-\alpha}(x) \} \mid X=x, \mathcal{D}_{n}),
	\end{align*}
	
\noindent	and,
	
	\begin{align*}
	& \mathbb{P}(Y \in \widetilde{C}_q^{\alpha}(x) \triangle \widetilde{C}^{\alpha}(x) \mid X=x, \mathcal{D}_{n}) \leq 2 \mid G^{*}(\widetilde{q}_{1-\alpha}(x),x)- G^{*}(q_{1-\alpha}(x),x)| = o_p(1),
	\end{align*}
	
	\noindent and as a consequence
	
	\begin{align*}
	\int_{\mathcal{X}} & \mathbb{P}(Y \in \widetilde{C}_q^{\alpha}(x) \triangle \widetilde{C}^{\alpha}(x) \mid X=x, \mathcal{D}_{n}) P_{X} (dx) = o_{p}(1).
	\end{align*}
	
\noindent	Finally, combining the prior four results, we infer that
	
	\begin{align*}
	\int_{\mathcal{X}} & \mathbb{P}(Y \in C^{\alpha}(x) \triangle \widetilde{C}^{\alpha}(x)\mid X=x, \mathcal{D}_{n}) P_{X}(dx) = o_{p}(1).
	\end{align*}

	\end{proof}


    



	
	
	
\end{document}